\documentclass[a4paper,11pt]{amsart}
\usepackage[top=1.25in, bottom=1.25in, left=1.25in, right=1.25in]{geometry}
\usepackage{graphicx}
\usepackage{amsfonts}
\usepackage{epsf}
\usepackage{amssymb}
\usepackage{amsmath}
\usepackage{amscd}
\usepackage{amsthm}
\usepackage{mathtools}
\usepackage{tikz}
\usetikzlibrary{
  cd,
  calc,
  positioning,
  arrows,
  decorations.pathreplacing,
  decorations.markings,
}
\usepackage{pdfpages}
\usepackage{setspace}
\usepackage{hyperref}
\usepackage[all]{xy}
\usetikzlibrary{matrix}
\usepackage{verbatim}
\usepackage{enumerate}
\usepackage{mathrsfs}
\usepackage{pinlabel}
\usepackage{lscape}
\usepackage{color}
\usepackage{enumitem}

\newtheorem{theorem}{Theorem}[section]

\newtheorem{proposition}[theorem]{Proposition}
\newtheorem{corollary}[theorem]{Corollary}
\newtheorem{lemma}[theorem]{Lemma}

\newtheorem{question}[theorem]{Question}

\theoremstyle{definition}
\newtheorem{definition}[theorem]{Definition}

\newtheorem{remark}[theorem]{Remark}

\newtheorem*{case2'}{Case 2$'$}
\newtheorem{theorem-named}{}
\newtheorem{theorem-labeled}{Theorem}
\newtheorem{definition-named}{}
\newtheorem{conjecture-named}{}
\newtheorem{case-named}{}

\numberwithin{equation}{section}

\newcommand{\R}{\mathbb{R}}
\newcommand{\C}{\mathbb{C}}

\def\N{\mathbb{N}}
\def\R{\mathbb{R}}

\def\C{\mathbb{C}}

\def\hat{\widehat}

\def\et{\quad\mbox{and}\quad}
\makeatletter
\newcommand{\red}{\textcolor{red}}

\usepackage{epsf,graphics,amssymb,amsmath,xcolor,asymptote,microtype}
\usepackage{hyperref}
\usepackage[justification=normal]{caption}
\usepackage[nameinlink,nosort]{cleveref}

\begin{asydef}
real strokewidth = 0.7;
transform myt = scale(1, 1.66);
pen braidpen = linewidth(strokewidth);

void mydraw(path p) {
    draw(myt * p, braidpen);
}

// x = some crossing
// 0 = the crossing in question
// o = no crossing

// do
// .....
// ...o.
// ..0..
// .x...
// .....

// except
// .x...
// ..x..
// ..0..
// .x...
// .....

// do
// ..x..
// ...x.
// ...x.
// ..0..
// .x...
// .....

// do
// .....
// .o...
// ..0..
// ...x.
// .....
real howFarOnTop(int[][] B, int x, int y) {
     if (( ((B[x + 1][y] == 0) && (B[x][y - 1] != 0) && (B[x - 1][y - 2] != 0) && ((B[x - 1][y + 1] == 0) || (B[x][y] == 0)))  || ((B[x + 1][y] != 0) && (B[x][y - 1] != 0) && (B[x - 1][y - 2] != 0) && (B[x + 1][y + 1] != 0) && (B[x][y + 2] != 0))) ||
         ( ((B[x][y] == 0) && (B[x + 1][y - 1] != 0) && (B[x + 2][y - 2] != 0) && ((B[x + 2][y + 1] == 0) || (B[x + 1][y] == 0)))  || ((B[x][y] != 0) && (B[x + 1][y - 1] != 0) && (B[x + 2][y - 2] != 0) && (B[x][y + 1] != 0) && (B[x + 1][y + 2] != 0))))
     {
        return 0.15;
        dot(myt * (x,y), scale(1.5) * green);
    }
    return 0;
}

real howFarOnBottom(int[][] B, int x, int y) {
     if (( ((B[x + 1][y - 1] == 0) && (B[x][y] != 0) && (B[x - 1][y + 1] != 0) && ((B[x - 1][y - 2] == 0) || (B[x][y - 1] == 0)))  || ((B[x + 1][y - 1] != 0) && (B[x][y] != 0) && (B[x - 1][y + 1] != 0) && (B[x + 1][y - 2] != 0) && (B[x][y - 3] != 0))) ||
         ( ((B[x][y - 1] == 0) && (B[x + 1][y] != 0) && (B[x + 2][y + 1] != 0) && ((B[x + 2][y - 2] == 0) || (B[x + 1][y - 1] == 0)))  || ((B[x][y - 1] != 0) && (B[x + 1][y] != 0) && (B[x + 2][y + 1] != 0) && (B[x][y - 2] != 0) && (B[x + 1][y - 3] != 0))))
        return -0.15;
    return 0;
}

path hidingDisk(int x, int y) {
    return shift(x - .5, y + .5) * scale(0.2) * inverse(myt) * unitcircle;
}

void drawToLeft(int[][] B, int x, int y, pen braidpen, bool gap) {
    pair outgoingAngle = ((B[x + 1][y - 1] == 0) || ((B[x + 1][y - 1] != 0) && (B[x][y + 1] != 0) && (B[x + 1][y + 2] != 0)) || ((B[x + 1][y - 1] != 0) && (B[x + 1][y - 2] != 0) && (B[x][y - 3] != 0))) ? up : up + left;
    pair incomingAngle = ((B[x - 1][y + 1] == 0) || ((B[x - 1][y + 1] != 0) && (B[x - 1][y + 2] != 0) && (B[x][y + 3] != 0)) || ((B[x - 1][y + 1] != 0) && (B[x][y - 1] != 0) && (B[x - 1][y - 2] != 0))) ? up : up + left;

    pair midpoint = (x - .5, y + .5);
    path p;
    if (howFarOnBottom(B, x, y) != 0)
        p = (x,y + howFarOnBottom(B, x, y) + howFarOnTop(B, x, y)){outgoingAngle}..(x - .2, y + .2)..midpoint..{incomingAngle}(x - 1, y + 1 + howFarOnTop(B, x - 1, y + 1) + howFarOnBottom(B, x - 1, y + 1));
    else if (howFarOnTop(B, x - 1, y + 1) != 0)
        p = (x,y + howFarOnBottom(B, x, y) + howFarOnTop(B, x, y)){outgoingAngle}..midpoint..(x - .8, y + .8)..{incomingAngle}(x - 1, y + 1 + howFarOnTop(B, x - 1, y + 1) + howFarOnBottom(B, x - 1, y + 1));
    else
        p = (x,y + howFarOnBottom(B, x, y) + howFarOnTop(B, x, y)){outgoingAngle}..midpoint..{incomingAngle}(x - 1, y + 1 + howFarOnTop(B, x - 1, y + 1) + howFarOnBottom(B, x - 1, y + 1));
    if (gap) {
        real[][] i = intersections(p, hidingDisk(x, y));
        draw(myt * subpath(p, 0, i[0][0]), braidpen);
        draw(myt * subpath(p, i[i.length - 1][0], 3), braidpen);
    } else
        draw(myt * p, braidpen);
}

void drawToRight(int[][] B, int x, int y, pen braidpen, bool gap) {
    pair outgoingAngle = ((B[x][y - 1] == 0) || ((B[x][y - 1] != 0) && (B[x + 1][y + 1] != 0) && (B[x][y + 2] != 0)) || ((B[x][y - 1] != 0) && (B[x][y - 2] != 0) && (B[x + 1][y - 3] != 0))) ? up : up + right;
    pair incomingAngle = ((B[x + 2][y + 1] == 0) || ((B[x + 2][y + 1] != 0) && (B[x + 2][y + 2] != 0) && (B[x + 1][y + 3] != 0)) || ((B[x + 2][y + 1] != 0) && (B[x + 1][y - 1] != 0) && (B[x + 2][y - 2] != 0))) ? up : up + right;

    pair midpoint = (x + .5, y + .5);
    path p;
    if (howFarOnBottom(B, x, y) != 0)
        p = (x,y + howFarOnBottom(B, x, y) + howFarOnTop(B, x, y)){outgoingAngle}..(x + .2, y + .2)..midpoint..{incomingAngle}(x + 1, y + 1 + howFarOnTop(B, x + 1, y + 1) + howFarOnBottom(B, x + 1, y + 1));
    else if (howFarOnTop(B, x + 1, y + 1) != 0)
        p = (x,y + howFarOnBottom(B, x, y) + howFarOnTop(B, x, y)){outgoingAngle}..midpoint..(x + .8, y + .8)..{incomingAngle}(x + 1, y + 1 + howFarOnTop(B, x + 1, y + 1) + howFarOnBottom(B, x + 1, y + 1));
    else
        p = (x,y + howFarOnBottom(B, x, y) + howFarOnTop(B, x, y)){outgoingAngle}..midpoint..{incomingAngle}(x + 1, y + 1 + howFarOnTop(B, x + 1, y + 1) + howFarOnBottom(B, x + 1, y + 1));
    if (gap) {
        real[][] i = intersections(p, hidingDisk(x + 1, y));
        draw(myt * subpath(p, 0, i[0][0]), braidpen);
        draw(myt * subpath(p, i[i.length - 1][0], 3), braidpen);
    } else
        draw(myt * p, braidpen);
}

void drawbraid(int[] b, int strands = -1, pen[] strandPens = {}) {
    // delete zeroes
    for (int i = 0; i < b.length;) {
        if (b[i] == 0)
            b.delete(i);
        else
            ++i;
    }

    // get right number of strands
    if (strands == -1) {
        int m = 1;
        for (int i = 0; i < b.length; ++i)
            if (abs(b[i]) > m)
                m = abs(b[i]);
        strands = m + 1;
    }

    // slide, prepare matrix
    int[][] B = array(strands + 3, array(b.length + 4, 0));
    for (int i = 0; i < b.length; ++i) {
        int x = abs(b[i]) + 1;
        int y = b.length + 3;
        while ((B[x - 1][y - 1] == 0) && (B[x][y - 1] == 0) && (B[x + 1][y - 1] == 0) && (y > 2))
            --y;
        B[x][y] = (b[i] > 0) ? 1 : -1;
    }

    // until where are there crossings?
    int highestY = 0;
    for (int y = 2; y < B[0].length; ++y) {
        bool thereAreCrossings = false;
        for (int x = 1; x < strands + 1; ++x)
            if (B[x][y] != 0) {
                thereAreCrossings = true;
                break;
            }
        if (! thereAreCrossings) {
            highestY = y;
            break;
        }
    }

    //grid
/*    for (int x = 1; x < strands + 1; ++x)
        draw(myt * ((x, 1)--(x, highestY)), lightgray);
    for (int y = 1; y < highestY; ++y)
        draw(myt * ((1, y)--(strands, y)), lightgray);
*/

    // follow the strands
    int[] strandNumbers;
    for (int x = 0; x < strands; ++x)
        strandNumbers.push(x);

    // draw
    for (int y = 2; y < highestY; ++y)
        for (int x = 1; x < strands + 1; ++x) {
            pen leftPen = braidpen;
            pen rightPen = braidpen;
            if (strandPens.length != 0) {
                leftPen = strandPens[strandNumbers[x - 1]];
                if (x != 1)
                    rightPen = strandPens[strandNumbers[x - 2]];
            }
            if (B[x][y] != 0) {
                drawToRight(B, x - 1, y, rightPen, (B[x][y] < 0));
                drawToLeft(B, x, y, leftPen, (B[x][y] > 0));

                int tmp = strandNumbers[x - 2];
                strandNumbers[x - 2] = strandNumbers[x - 1];
                strandNumbers[x - 1] = tmp;
            }
            else if (B[x + 1][y] == 0)
                draw(myt * ((x, y + howFarOnTop(B, x, y))--(x, y + 1 + howFarOnBottom(B, x, y + 1))), leftPen);
        }

    // cut off above and below
    clip(myt * ((0,2)--(2 + strands, 2)--(2 + strands, highestY)--(0, highestY)--cycle));
} 

void drawbraid(string s, int strands = -1, pen[] strandPens = {}) {
    int[] b = {};
    for(int i = 0; i < length(s); ++i) {
        int c = 0;
        int a = ascii(substr(s, i));
        if ((a >= 97) && (a <= 122))
            c = a - 96;
        else if ((a >= 65) && (a <= 90))
            c = -(a - 64);
        if (c != 0)
            b.push(c);
    }
    drawbraid(b, strands, strandPens);
}

//unitsize(1cm);
//drawbraid("abcddcbaabcd");
\end{asydef}

\begin{asydef}
unitsize(2.5mm);
strokewidth = 0.6;
\end{asydef}

\begin{document}

%\captionsetup{justification=centering}

\title{Genus one cobordisms between torus knots}

\author{Peter Feller}
\address{ETH Zurich, R\"amistrasse 101, 8092 Zurich, Switzerland}
\email{peter.feller@math.ch}
\urladdr{people.math.ethz.ch/~pfeller/}
\author{JungHwan Park}
\address{School of Mathematics, Georgia Institute of Technology, Atlanta, GA, USA}
\email{junghwan.park@math.gatech.edu }
\urladdr{people.math.gatech.edu/~jpark929/}
\def\subjclassname{\textup{2010} Mathematics Subject Classification}
\expandafter\let\csname subjclassname@1991\endcsname=\subjclassname
\expandafter\let\csname subjclassname@2000\endcsname=\subjclassname
 % Cobordism and concordance (in low dimension)
  %  57Q60; % Cobordism and concordance (in high dimension)
%  57M07, % Topological methods in group theory


\begin{abstract}
We determine the pairs of torus knots that have a genus one cobordism between them, with one notable exception. This is done by combining obstructions using $\nu^+$ from the Heegaard Floer knot complex and explicit constructions of cobordisms. As an application,
we determine the pairs of torus knots related by a single crossing change. Also, we determine the pairs of Thurston-Bennequin number maximizing Legendrian torus knots that have a genus one exact Lagrangian cobordism, with one exception.
\end{abstract}

\maketitle

\section{Introduction}\label{sec:intro}

Let $K$ and $J$ be knots---smooth non-empty connected oriented 1-submanifolds of the $3$-sphere $S^3$. The \emph{cobordism distance} between $K$ and $J$, denoted by $d(K,J)$, is defined to be the smallest integer that arises as the genus of a smoothly embedded oriented surface in $S^3\times[0,1]$ with boundary (appropriately oriented) $J\times \{0\}$ and $K\times \{1\}$. Equivalently, $d(K,J)$ maybe defined as $g_4(J\#-K)$, where $g_4(K)$ denotes the smooth 4-ball genus of a knot $K$, $\#$ denotes connected sum of knots, and $-K$ denotes the knot obtained from $K$ by mirroring and reversing orientation. The cobordism distance and the $4$-ball genus of knots are generally hard to determine, but for torus knots the $4$-ball genus is known by the local Thom conjecture~\cite{Kronheimer-Mrowka:1993-1}.
However, the cobordism distance between two non-trivial torus knots %that are not unknots
is, in general, not understood beyond some partial progress.
Besides this being a natural next case% (after the distance between the unknot and a torus knot)
, the determination of the cobordism distance between torus knots is of interest because of connections to questions about the existence of deformations of singularities of plane curves~\cite{Arnold_normalforms,Borodzik-Livingston:2016-1, Feller_15_MinCobBetweenTorusknots}.

We describe what is known.
Non-isotopic torus knots have non-zero cobordism distance;
in fact, non-trivial positive torus knots are linearly independent in the concordance group~\cite{Litherland_79_SignaturesOFIteratedTorusKnots}. Trotter's classical knot signature~\cite{Trotter_62_HomologywithApptoKnotTheory} and the Tristram-Levine signatures~\cite{Tristram:1969-1,Levine:1969-1}
allow to determine the cobordism distance of most torus knots of two fixed braid indices up to a constant~\cite{Baader_ScissorEq}.
The modern Heegaard Floer concordance invariants $\nu^+$~\cite{Hom-Wu:2016-1} and $\Upsilon$~\cite{OSS_2014}
lead to better bounds on cobordism distance depending on the braid indices~\cite{Bodnar-Celoria-Golla:2017-1,FellerKrcatovich_16_OnCobBraidIndexAndUpsilon}.
And for small braid indices, these invariants allow to compute the cobordism distance completely~\cite{Feller_15_MinCobBetweenTorusknots,BaaderFellerLewarkZentner_16}.

In this text, we determine which pairs of torus knots have cobordism distance one, with one exception.
For the exceptional pair, we know the distance is at most two. This brings us to ask:

\begin{question}\label{Q:genus1cob}
Is the cobordism distance between $T_{3,14}$ and $T_{5,8}$ one?\end{question}
All other pairs are covered by our main result:
\begin{theorem}\label{thm:main} Let $\{T_{p,q}, T_{p',q'} \}$ be any pair of non-trivial positive torus knots; that is, both $\{p,q\}$ and $\{p',q'\}$ are two pairs of coprime integers larger than or equal to $2$. If the pair is one of the following
\begin{enumerate}[font=\upshape]
\item \label{fam1}$\{ T_{2,2n+1}, T_{2,2n+3}\}$ for $n\geq1$,
\item \label{fam2}$\{ T_{3,3n+1}, T_{3,3n+2}\}$ for $n\geq1$,
\item \label{fam3}$\{ T_{3n+1,9n+6}, T_{3n+2,9n+3}\}$ for $n\geq1$,
\item \label{fam4}$\{ T_{2n+1,4n+6}, T_{2n+3,4n+2}\}$ for $n\geq1$,
\item \label{fam5}$\{ T_{2,2n-3}, T_{3,n}\}$ for $n\in\{4,5,7,8\}$,
\item \label{fam6}$\{ T_{2,7}, T_{3,4}\}$, $\{ T_{2,9}, T_{3,5}\}$, $\{ T_{2,11}, T_{4,5}\}$, $\{ T_{3,7}, T_{4,5}\}$, $\{ T_{3,10}, T_{4,7}\}$, $\{ T_{4,9}, T_{5,7}\}$,
\end{enumerate}
then there exists a genus one cobordism between them $(i.e.~d(T_{p,q}, T_{p',q'})=1)$.

If the pair is not one of \eqref{fam1}--\eqref{fam6} nor
\begin{enumerate}[font=\upshape]
\setcounter{enumi}{6}
\item \label{fam7}$\{ T_{3,14}, T_{5,8}\}$,
\end{enumerate}
then there is no genus one cobordism between them $(i.e.~d(T_{p,q}, T_{p',q'}) \geq 2)$.\end{theorem}

Note that the cobordism distance between two negative torus knots is equal to the cobordism distance between their mirrors, which are positive torus knots.
Also, by using the Ozsv{\'a}th-Szab{\'o} $\tau$ invariant \cite{Ozsvath-Szabo:2003-3}, it is straightforward to verify that the cobordism distance between a positive torus knot and a negative torus knot is equal to the sum of their $4$-ball genera.
Hence, considering pairs of positive torus knots, as is done in Theorem~\ref{thm:main}, suffices to classify all pairs of torus knots with cobordism distance one.

The cobordisms in Theorem~\ref{thm:main} are constructed using positive braid manipulations.
Before we describe how we obstruct the existence of genus one cobordisms between pairs of torus knots %not in $\eqref{fam1}$--$\eqref{fam7}$
using $\nu^+$ from Heegaard Floer theory, we discuss an application to an unknotting question.

\subsection*{Torus knots of Gordian distance one}
Given two knots $K$ and $J$, their \emph{Gordian distance} is defined to be the minimal number of crossing changes needed to get from $K$ to $J$ (see e.g.~\cite{Murakami:1985-1}).
As a consequence of Theorem~\ref{thm:main} (and its proof), we find which pairs of torus knots have Gordian distance one:

\begin{corollary}\label{cor:3}
Let $T_{p,q}$ and $T_{p',q'}$ be distinct positive torus knots. The knots $T_{p,q}$ and $T_{p',q'}$ have Gordian distance one $($i.e.~one can be turned into the other by one crossing change$)$ if and only if $\{T_{p,q}, T_{p',q'} \}$ is one of the following:

\begin{enumerate}[font=\upshape]
\item $\{ T_{2,2n+1}, T_{2,2n+3}\}$ for $n\geq0$,
\item $\{ T_{3,3n+1}, T_{3,3n+2}\}$ for $n\geq1$,
\item $\{ T_{2,5}, T_{3,4}\}$, $\{T_{2,7}, T_{3,5}\}$.
\end{enumerate}
\end{corollary}

Here is a bit of context on this result.
The study of the Gordian distance goes back to Wendt's considerations on the unknotting number  or Gordian number $u(K)$ of a knot---the Gordian distance of a knot to the unknot~\cite{Wendt:1937-1}.
The Gordian graph is the graph that has isotopy classes of knots as vertices and an edge between any two vertices with Gordian distance one. One way to phrase Corollary~\ref{cor:3} is the following.
The induced subgraph on isotopy classes of torus knots is given by edges between the pairs described in Corollary~\ref{cor:3} and edges between their mirrors.
Note that the study of the Gordian distance between torus knots (the distance in the Gordian Graph between vertices given by torus knots) turns out to be very subtle; see e.g.~\cite{GambaudoGhys_BraidsSignatures}. We only treat the distance one case.

The point of Corollary~\ref{cor:3} is not that the given pairs of torus knots are related by a crossing change; this is well-known and fits well with a related concept, the existence of so called $\delta$-constant deformations between the corresponding simple singularities of plane curves (see~\cite{Borodzik-Livingston:2016-1} for how $\delta$-constant deformations yield unknottings up to concordance).
Instead, we use Theorem~\ref{thm:main} to exclude most pairs of torus knots to have Gordian distance one and then discuss the remaining pairs using the same invariant as for Theorem~\ref{thm:main} and Tristram-Levine signatures; see Section~\ref{sec:Gordian}.

\subsection*{Obstructing cobordisms using $\nu^+$}

As an obstruction to having cobordisms of genus one, we use $\nu^+$---a positive integer valued knot invariant defined by Hom and Wu \cite{Hom-Wu:2016-1} satisfying
$$\nu^+(J\# -K)\leq g_4(J\# -K)=d(K,J)$$ for all knots $K$ and $J$. Using a recipe (see \cite[Theorem 1.1]{Bodnar-Celoria-Golla:2017-1}) to determine $\nu^+$ for connected sums of torus knots in terms of their semi-groups (invariants more generally associated with knots of plane curve singularities), we establish the following.

\begin{proposition}\label{prop:nu}
Let $T_{p,q}$ and $T_{p',q'}$ be distinct non-trivial positive torus knots. We have
\[ \max \{ \nu^+(T_{p,q} \# -T_{p',q'}), \nu^+(T_{p',q'} \# -T_{p,q}) \} \leq 1\]
if and only if $\{T_{p,q}, T_{p',q'} \}$ is one of the pairs in the families \eqref{fam1}--\eqref{fam7} in Theorem~\ref{thm:main}.

In particular, if a pair of distinct non-trivial positive torus knots is not in the families \eqref{fam1}--\eqref{fam7}, then they do not bound a genus one cobordism.
\end{proposition}

\subsection*{Obstructing cobordisms using Tristram-Levine signatures}
At this point some readers might wonder why the authors are using $\nu^+$ instead of say classical invariants, such as the Tristram-Levine signatures, which also obstruct the existence of cobordisms in the topological category. However, Tristam-Levine signatures are not sufficient to obtain Theorem~\ref{thm:main}. Indeed, the pair of torus knots $T_{5,17}$ and $T_{7,12}$ has cobordism distance at least two (by Theorem~\ref{thm:main}), but the lower bound
\[\max_{\omega\in S^1\text{ regular for $K$ and $J$}}\tfrac{1}{2}\left|\sigma_{\omega}(K)-\sigma_{\omega}(J)\right|\leq d(K,J)
%\quad\text{for all knots }K, J
\]
evaluated for $K=T_{5,17}$ and $J=T_{7,12}$ is
\[\max_{\omega\in S^1 \text{ regular for $K$ and $J$}}\tfrac{1}{2}\left|\sigma_{\omega}(T_{5,17})-\sigma_{\omega}(T_{7,12})\right|=1.\]

%Since $\nu^+$ obstructs smooth cobordisms (not locally flat ones)
As a consequence, one may wonder the following about differences between the smooth and topological category:
\begin{question}Does there exist a pair of positive torus knots where there exists a locally flat genus one cobordism between them but not a smooth one?\end{question}
In addition to $T_{5,17}$ and $T_{7,12}$, the torus knots $T_{5,18}$ and $T_{8,11}$ constitute another pair that has smooth cobordism distance at least two by Theorem~\ref{thm:main}, but the authors do not know whether the topological cobordism distance is one.

Even for families of torus knots where experimental evidence suggests that Tristram-Levine signatures could suffice as obstructions, the authors do not see how to achieve this, given that the formulas for Tristram-Levine signatures of torus knots are unwieldy. %, and the author are not sure how one would approach calculating them for the given infinite families.

Finally, as a side remark, we note that Theorem~\ref{thm:main} could not have been proved using $\Upsilon$ (a lower bound for $\nu^+$ introduced by Ozsv\'{a}th, Stipsicz, and Szab\'{o}~\cite{OSS_2014} that has the benefit of being a (collection) of concordance homomorphism(s)). Indeed, $\Upsilon$ does not obstruct the pair $T_{5,17}$ and $T_{7,12}$ from having cobordism distance one.%, as any $\Upsilon$-enthusiast may check. 
\subsection*{Positive braids, decomposable Lagrangian cobordisms, and Hopf plumbing}
For a positive integer $n$, the standard group presentation for Artin's \emph{braid group on $n$ strands}~\cite{Artin_TheorieDerZoepfe}, denoted by $B_n$, is given by generators $a_1,\ldots, a_{n-1}$ (known as Artin generators) subject to the \emph{braid relations}
\[a_ia_{i+1}a_i=a_{i+1}a_{i}a_{i+1}\text{ for }\; 1\leq i\leq n-2\et
 a_ia_j=a_ja_i \text{ for }\; |i-j|\geq 2.\] We refer to~\cite{Birman_74_BraidsAMSStudies} for details on braids, including the notions of braid closures and braid diagrams (as used in Figures~\ref{fig:beta} to~\ref{fig:45to37}).

Given a positive braid $\beta$ (that is, a braid that is the product of positive powers of Artin generators), its closure has an associated Legendrian representative $\Lambda_\beta$ in the standard contact structure of $\mathbb{R}^3$, which is well-defined up to Legendrian isotopy and maximizes the Thurston-Bennequin number (see Section~\ref{sec:lag} for details).
In fact, two positive braids with the same closure yield Legendrian isotopic Legendrian representatives~\cite[Corollary~1.13]{Etnyre-VanHornMorris:2011-1}. Therefore, one has a canonical Legendrian representative for every link that is the closure of a positive braid. 

Our constructions of cobordisms come from manipulating positive braids as follows: cyclic permutation of braid words (corresponding to an isotopy), positive Markov stabilization (corresponding to an isotopy), and deletion of a positive generator (corresponding to a 1-handle attachment). It turns out that these, in fact, correspond to \emph{decomposable Lagrangian cobordisms}, which are exact Lagrangian cobordisms (see Definition~\ref{Def:LagCob}) that are broken into elementary pieces associated to Legendrian isotopies, pinches, and births. We note that Lagrangian cobordisms are directed: existence of a Lagrangian cobordism from $\Lambda$ to
$\Lambda'$, does not imply that there exists a Lagrangian cobordism from $\Lambda'$ to
$\Lambda$.
\begin{lemma}\label{lemma:posbraidstodecompcobs} 
Let $\beta$ and $\beta'$ be positive braids such that there is a sequence of positive braids $\beta_0,\beta_1,\ldots,\beta_n$ with $\beta_0=\beta$ and $\beta_n=\beta'$ such that, for all $j\in\{0,1,\ldots,n-1\}$, $\beta_{j+1}$ and $\beta_{j}$ are related by one of the following:
\begin{itemize}[font=\upshape]
\item[(i)] $\beta_{j+1}$ is obtained from $\beta_{j}$ by cyclic permutation,
  \item[(ii)] $\beta_{j+1}$ is obtained from $\beta_{j}$ by positive Markov stabilization or destabilization, or
  \item[(iii)] $\beta_{j+1}=\beta_{j}a$, where $a$ is one of the standard positive Artin generators.
\end{itemize}

\noindent Then there exists a decomposable Lagrangian cobordism from $\Lambda_{\beta}$ to $\Lambda_{\beta'}$.\end{lemma}
We suspect that Lemma~\ref{lemma:posbraidstodecompcobs} is well-known to experts, but we provide a proof for completeness.
\begin{remark}We note that Lemma~\ref{lemma:posbraidstodecompcobs} fits well with the following analog for quasipositive braids. Rudolph showed that every closure of a quasipositive braid arises as the transverse intersection of a smooth complex curve in $\C^2$ with a round 3-sphere centered at the origin~\cite{Rudolph_83_AlgFunctionsAndClosedBraids}. Furthermore, it turns out that if one can get from one closure of a quasipositive braid to another using the analog of the above (i), (ii), and (iii) for quasipositive generators, then their closures are related by an algebraic cobordism; that is, there exist two round 3-spheres in $\C^2$ and a smooth complex curve intersecting them transversally such that the intersections are the two closures, respectively~\cite[Lemma~6]{Feller_15_MinCobBetweenTorusknots}. As a consequence, one can paraphrase Lemma~\ref{lemma:posbraidstodecompcobs} as follows: if one obtains one positive braid from the other by a sequence of (i), (ii), (iii), then the corresponding algebraic cobordism is in fact given as a decomposable Lagrangian one. 
\end{remark}

As a corollary of the proof of Theorem~\ref{thm:main}, we have the following. For $p,q\geq 1$ coprime, let $\Lambda_{p,q}$ denote the Legendrian knot $\Lambda_{\beta_{p,q}}$ that is associated with the positive $p$-braid
\[\beta_{p,q}=(a_1\cdots a_{p-1})^q\in B_p\] with closure the torus knot $T_{p,q}$.
For additional motivation of the study of $\Lambda_{p,q}$, we note that $\Lambda_{p,q}$ is the unique Legendrian representative of $T_{p,q}$ that maximizes the Thurston-Bennequin number~\cite{Etnyre-Honda:2001-1}.

\begin{theorem}\label{thm:genus1deccob} Let $(\Lambda_{p,q}, \Lambda_{p',q'})$ be any ordered pair of Thurston-Bennequin number maximizing Legendrian knots that represent non-trivial positive torus knots $T_{p,q}$ and $T_{p',q'}$, respectively. If the ordered pair is one of the following
\begin{enumerate}[font=\upshape]
\item \label{lfam1}$( \Lambda_{2,2n+1}, \Lambda_{2,2n+3})$ for $n\geq1$,
\item \label{lfam2}$( \Lambda_{3,3n+1}, \Lambda_{3,3n+3})$ for $n\geq1$,
\item \label{lfam3}$( \Lambda_{3n+1,9n+6}, \Lambda_{3n+2,9n+3})$ for $n\geq1$,
\item \label{lfam4}$( \Lambda_{2n+1,4n+6}, \Lambda_{2n+3,4n+2})$ for $n\geq1$,
\item \label{lfam5}$( \Lambda_{2,2n-3}, \Lambda_{3,n})$ for $n\in\{4,5,7,8\}$,
\item \label{lfam6} $( \Lambda_{2,11}, \Lambda_{4,5})$, 
\end{enumerate}
then there exists a genus one decomposable Lagrangian cobordism from $\Lambda_{p,q}$ to $\Lambda_{p',q'}$. 

If the pair is not one of \eqref{lfam1}--\eqref{lfam6} nor
\begin{enumerate}[font=\upshape]
\setcounter{enumi}{6}
\item \label{lfam7}$( \Lambda_{3,14},  \Lambda_{5,8})$,
\end{enumerate}
then there is no genus one exact Lagrangian cobordism from $\Lambda_{p,q}$ to  $\Lambda_{p',q'}$.

Furthermore, for all pairs $T_{p,q}$ and $T_{p',q'}$ that have the same $g_4$ and that have cobordism distance one by Theorem~\ref{thm:main}, the following holds. There exists a 2-component Legendrian link $L$ and two decomposable Lagrangian cobordisms, one from $\Lambda_{p,q}$ to $L$ and one from $\Lambda_{p',q'}$ to $L$,\footnote{In fact, from the proof it will be clear, that also the following holds. There exists a 2-component Legendrian link $L'$ and two decomposable Lagrangian cobordisms, one from  $L'$ to $\Lambda_{p',q'}$ and one from $L'$ to $\Lambda_{p,q}$.} that glue together to give a genus one cobordism from $T_{p,q}$ and $T_{p',q'}$.
\end{theorem}

\begin{remark}[Different interpretations of the constructions via positive braids]\label{rem:plumbings}
The cobordisms between torus knots from~\cite{Baader_ScissorEq,Feller_15_MinCobBetweenTorusknots,BaaderFellerLewarkZentner_16} (compare with the second paragraph above) are constructed using (i), (ii), (iii) from Lemma~\ref{lemma:posbraidstodecompcobs}. Consequently, these cobordisms are in fact (isotopic to) decomposable Lagrangian cobordisms.

Cobordisms constructed using (i), (ii), (iii) from Lemma~\ref{lemma:posbraidstodecompcobs} can be understood in terms of the corresponding fiber surfaces in $S^3$ as follows (see also~\cite[Remark~19]{Feller_15_MinCobBetweenTorusknots} for more details). Let $\beta$ and $\beta'$ be positive braids as in Lemma~\ref{lemma:posbraidstodecompcobs} and additionally assume that they have non-split closures (e.g.~the closures are knots). The fiber surface of the closure of $\beta'$ is obtained by positive Hopf plumbing of the fiber surface of the closure of $\beta$. Consequently, the cobordisms constructed in this paper correspond to plumbing (and deplumbing) of positive Hopf bands.
\end{remark}

As a first step towards a negative answer to Question~\ref{Q:genus1cob}, we wonder whether one can use contact invariants to provide a negative answer to the following question.
\begin{question}\label{Q:genus1deccob} Does there exist a decomposable Lagrangian cobordism from $\Lambda_{3,14}$ to~$\Lambda_{5,8}$?
\end{question}
In light of Remark~\ref{rem:plumbings}, Theorem~\ref{thm:main} and its proof determine for which pairs of torus knots their corresponding fiber surfaces are related by plumbing or deplumbing two positive Hopf bands, with one exception. We wonder whether there are obstructions that would answer the following question in the negative. Can the fiber surface of $T_{5,8}$ be obtained by plumbing two positive Hopf bands to the fiber surface of $T_{3,14}$?

\subsection*{Acknowledgments}
This project started when both authors were at the Max-Planck-Institute for Mathematics in Bonn. The project took its final form when the first author visited the second author at Georgia Tech. We would like to thank both the MPIM and Georgia Tech for providing an excellent environment for research. We thank John Etnyre for his interest in the project and specifically suggesting the consideration of decomposable cobordisms in general and Lemma~\ref{lemma:posbraidstodecompcobs} in particular. We thank Lukas Lewark for being inquisitive about bounds cobordism distance bounds from Tristram-Levine signatures. We also thank Marco Golla, Jennifer Hom, Allison Miller, Hyun Ki Min, and Lee Rudolph for helpful conversations and remarks. {We are also grateful to the anonymous referee for their detailed and thoughtful suggestions.}
%%%%%%%%%%%%%%%%%%%%%%%%%%%%%%%%%%%%%%%%%%
\section{Proof of Theorem~\ref{thm:main} Part 1: finding cobordisms}\label{sec:proofThm1.2}
In this article, we work in the smooth category, and all manifolds are oriented. To find the genus one cobordism establishing the first part of Theorem~\ref{thm:main}, we use previous constructions together with three new explicitly constructed cobordisms that realize the cobordism distance. We restate the first part of Theorem~\ref{thm:main}, about the existence of genus one cobordisms, in a separate proposition.
\begin{proposition}\label{prop:existenceofcobs}
Let $\{T_{p,q}, T_{p',q'} \}$ be any pair of non-trivial positive torus knots. If the pair is one of those appearing in \eqref{fam1}--\eqref{fam6} in Theorem~\ref{thm:main},
then there exists a genus one cobordism between them $(i.e.~d(T_{p,q}, T_{p',q'})=1)$.
\end{proposition}
All cobordisms constructed below will be given as a composition of $1$-handles (saddles) corresponding to a saddle move on knots and links. Here a \emph{saddle move} is understood to be the operation of changing a link in a $3$--ball as diagrammatically described on the left-hand side in Figure~\ref{fig:saddlemove}.
\begin{figure}[h]
\centering
\def\svgscale{1.9}
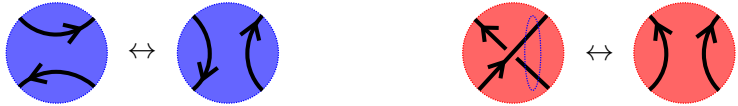
\caption{Left: A saddle move. Right: A smoothing of a crossing, which is obtained by applying a saddle move in the ball that is a small neighborhood of the obvious spanning disk for the blue circle pictured.}
\label{fig:saddlemove}
\end{figure}
In fact, all saddles moves will be given as \emph{smoothing a crossing} (respectively its inverse: adding a crossing) in a diagram of a link, which is another diagrammatic representation of a saddle move; see right-hand side in Figure~\ref{fig:saddlemove}.
And in turn, all our smoothing and adding of crossings, will correspond to deleting or adding a generator in a positive braid word. %For this we briefly set up notations for braids.

\begin{proof}[Proof of Proposition~\ref{prop:existenceofcobs}]
For torus knots $T_{p,q}$ and $T_{p',q'}$ with $p+p'\leq 6$, the cobordism distance has been determined; see~\cite[Corollary~3]{Feller_15_MinCobBetweenTorusknots}. In particular, the pairs of knots in the families
\eqref{fam1}, \eqref{fam2}, and \eqref{fam5} have cobordism distance one.
Furthermore, this also shows that half of the pairs of \eqref{fam6} have cobordism distance one; namely, $\{ T_{2,7}, T_{3,4}\}$, $\{ T_{2,9}, T_{3,5}\}$, and $\{ T_{2,11}, T_{4,5}\}$.
However, we note that these cobordisms were most likely all previously known to experts. For example, for all pairs of torus knots where both knots correspond to simple singularities of plane curves, the cobordism distance is (implicitly) determined by Arnold in his classification of adjacency of simple singularities~\cite{Arnold_normalforms}. For the families \eqref{fam1} and \eqref{fam2}, it is even known that these pairs of knots have Gordian distance one (that is, one can be turned into the other by one crossing change); compare with the proof of Corollary~\ref{cor:3} below.

To establish cobordism distance one for the families \eqref{fam3} and \eqref{fam4}, we recall the following result of Baader~\cite[Proposition~1]{Baader_ScissorEq}. Let $a,b,c$ be pairwise coprime positive integers. Then there exists a cobordism of genus $\frac{|b-a|(c-1)}{2}$ between the torus knots $T_{a,bc}$ and $T_{b,ac}$. Setting $a=3n+1$, $b=3n+2$, and $c=3$ for the family \eqref{fam3} and $a=2n+1$, $b=2n+3$, and $c=2$ for the family \eqref{fam4}, yields the genus one cobordisms. We make a remark on Baader's construction of these cobordisms, which will only become relevant in the proof of Corollary~\ref{cor:3}.

\begin{remark}\label{rem:abc}
 In fact (for this we assume w.l.o.g. that $a\leq b$), inspecting Baader's proof shows that he writes $T_{b,ac}$ as the closure of a positive braid given by a specific positive braid word from which one can delete $(b-a)(c-1)$ positive generators to obtain a positive braid with closure $T_{a,bc}$. This corresponds to smoothing $(b-a)(c-1)$ crossings in a standard diagram of the closure of the braid, and, thus, gives a genus $\frac{(b-a)(c-1)}{2}$ cobordism given as the concatenation of $(b-a)(c-1)$ 1-handles.

Similarly, all the above cobordisms that are coming from~\cite[Corollary~3]{Feller_15_MinCobBetweenTorusknots} are obtained as follows. The larger genus torus knot is realized as the closure of a positive braid from which two positive generators are deleted to obtain a positive braid with closure the other torus knot.
\end{remark}

It remains to discuss $\{ T_{3,7}, T_{4,5}\}$, $\{ T_{3,10}, T_{4,7}\}$, and $\{ T_{4,9}, T_{5,7}\}$ from \eqref{fam6}.

\noindent\textbf{Construction of a genus one cobordism between $T_{4,9}$ and $T_{5,7}$:}
We view $T_{5,7}$ as the closure of the 8-braid \[\beta\coloneqq(a_1a_2a_3a_4a_5a_6a_7)(a_1a_2a_3a_4a_5a_6)^{4},\]
which is obtained from the standard $7$-braid $(a_1a_2a_3a_4a_5a_6)^{5}$ with closure $T_{5,7}$ by one positive Markov stabilization.

Denote by $\Diamond$ the $8$-braid \[\Diamond\coloneqq a_4a_5a_6a_7a_3a_4a_5a_6a_2a_3a_4a_5a_1a_2a_3a_4,\] with the following commutation property:
\begin{equation}\label{eq:diamond}
\Diamond a_i=a_{i-4} \Diamond\et\Diamond a_{i-4}=a_{i}\Diamond\quad\forall i\in\{5,6,7\}.
\end{equation}
Applying braid relations one finds the following equalities of braids in the $8$-stranded braid group (see Left of Figure~\ref{fig:beta}):
\begin{align*}
  \beta=&a_1a_2a_3a_1a_2a_1\Diamond
  a_6a_5a_6(a_1a_2a_3a_4a_5a_6)=a_4a_1a_2a_3a_1a_2a_1\Diamond
  a_6a_5a_6(a_1a_2a_3a_4a_5).\end{align*}
By applying a cyclic permutation (a special case of conjugation, which preserves the closure; see Middle of Figure~\ref{fig:beta}) to the above positive word representing $\beta$ yields
\begin{align*}
  \beta'\coloneqq&a_1a_2a_3a_1a_2a_1\Diamond
  a_6a_5a_6(a_1a_2a_3a_4a_5)a_4=a_1a_2a_3a_1a_2a_1\Diamond a_6
  a_6a_5a_6(a_1a_2a_3\red{a_4}a_5),\end{align*}
where the equalities are again obtained by braid relations (see Right of Figure~\ref{fig:beta}).
\begin{figure}[h]
\centering

\begin{asy}
pen r = gray(0.7);
filldraw(myt * shift(6.5, 14 + 1.5) * inverse(myt) * scale(0.85) * (((N..(0.4*(N+E))..E) & (E..(0.4*(E+S))..S) & (S..(0.4*(S+W))..W) & (W..(0.4*(W+N))..N))--cycle), r, white);
pen g = rgb(0.2, 1, 0.2);
filldraw(myt * shift(4.5, 8.5) * inverse(myt) * scale(3.5,5.5) * (((N..(0.5*(N+E))..E) & (E..(0.5*(E+S))..S) & (S..(0.5*(S+W))..W) & (W..(0.5*(W+N))..N))--cycle), g, white);
drawbraid("abcdefgabcdefabcdefabcdefabcdef");
\end{asy}
\quad\raisebox{3.2cm}{$=$}
\begin{asy}
pen r = gray(0.7);
filldraw(myt * shift(4.5, 1 + 1.5) * inverse(myt) * scale(0.85) * (((N..(0.4*(N+E))..E) & (E..(0.4*(E+S))..S) & (S..(0.4*(S+W))..W) & (W..(0.4*(W+N))..N))--cycle), r, white);
pen g = rgb(0.2, 1, 0.2);
filldraw(myt * shift(4.5, 8.5) * inverse(myt) * scale(3.5,5.5) * (((N..(0.5*(N+E))..E) & (E..(0.5*(E+S))..S) & (S..(0.5*(S+W))..W) & (W..(0.5*(W+N))..N))--cycle), g, white);
drawbraid("dabcabadefgcdefbcdeabcdfefabcde");
\end{asy}
\quad\raisebox{3.2cm}{$\overset{\text{cyc.~per}}{\longleftrightarrow}$}
\begin{asy}
pen r = gray(0.7);
filldraw(myt * shift(4.5, 14 + 1.5) * inverse(myt) * scale(0.85) * (((N..(0.4*(N+E))..E) & (E..(0.4*(E+S))..S) & (S..(0.4*(S+W))..W) & (W..(0.4*(W+N))..N))--cycle), r, white);
pen g = rgb(0.2, 1, 0.2);
filldraw(myt * shift(4.5, 8.5) * inverse(myt) * scale(3.5,5.5) * (((N..(0.5*(N+E))..E) & (E..(0.5*(E+S))..S) & (S..(0.5*(S+W))..W) & (W..(0.5*(W+N))..N))--cycle), g, white);
drawbraid("abcabadefgcdefbcdeabcdfefabcded");
\end{asy}
\quad\raisebox{3.2cm}{$=$}
\begin{asy}
pen r = gray(0.7);
pen q = rgb(0.9, 0.1, 0.1);
filldraw(myt * shift(6.5, 11 + 1.5) * inverse(myt) * scale(0.85) * (((N..(0.4*(N+E))..E) & (E..(0.4*(E+S))..S) & (S..(0.4*(S+W))..W) & (W..(0.4*(W+N))..N))--cycle), r, white);
filldraw(myt * shift(4.5, 13 + 1.5) * inverse(myt) * scale(0.8) * unitcircle, q, white);
filldraw(myt * shift(7.5, 12 + 1.5) * inverse(myt) * scale(0.8) * unitcircle, q, white);
pen g = rgb(0.2, 1, 0.2);
filldraw(myt * shift(4.5, 8.5) * inverse(myt) * scale(3.5,5.5) * (((N..(0.5*(N+E))..E) & (E..(0.5*(E+S))..S) & (S..(0.5*(S+W))..W) & (W..(0.5*(W+N))..N))--cycle), g, white);
drawbraid("abcabadefgcdefbcdeabcdffefabcde");
\end{asy}

\captionof{figure}{Left: Braid isotopy of $\beta$ (first equality: the gray marked crossing is isotoped down).
\\
Middle: Cyclical permutation from $\beta$ to $\beta'$ (the gray marked crossing is removed on the bottom and added to the top).
\\Right: Braid isotopy of $\beta'$ (last equality: the gray marked crossing is isotoped down).
The crossing that is deleted and the place where a crossing is added to find $\beta'$ are marked red.}
\label{fig:beta}
\end{figure}
  Next, we observe that by deleting the last occurrence of $a_4$ (red above, and red-marked crossing in Figure~\ref{fig:beta}) and adding one $a_7$ (red below), we can turn $\beta'$ into
\[\beta''\coloneqq  a_1a_2a_3a_1a_2a_1\Diamond a_6
  a_6\red{a_7}a_5a_6(a_1a_2a_3a_5).\]
See Left of Figure~\ref{fig:beta''} for a depiction of $\beta''$.
\begin{figure}[h]
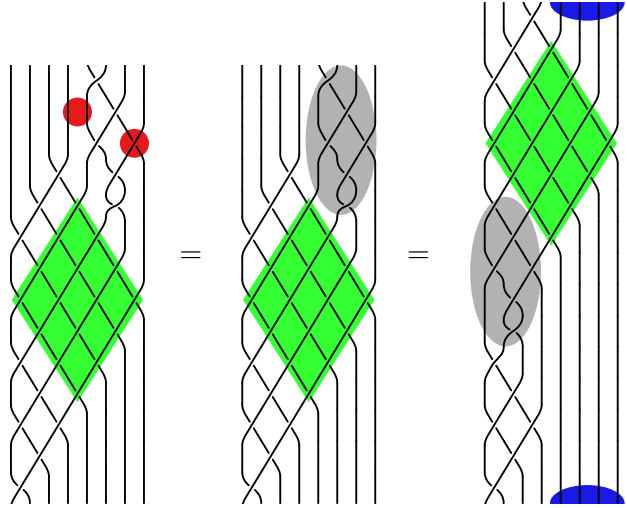

\centering

\begin{asy}
pen q = rgb(0.9, 0.1, 0.1);
filldraw(myt * shift(4.5, 13 + 1.5) * inverse(myt) * scale(0.8) * unitcircle, q, white);
filldraw(myt * shift(7.5, 12 + 1.5) * inverse(myt) * scale(0.8) * unitcircle, q, white);
pen g = rgb(0.2, 1, 0.2);
filldraw(myt * shift(4.5, 8.5) * inverse(myt) * scale(3.5,5.5) * (((N..(0.5*(N+E))..E) & (E..(0.5*(E+S))..S) & (S..(0.5*(S+W))..W) & (W..(0.5*(W+N))..N))--cycle), g, white);
drawbraid("abcabadefgcdefbcdeabcdffgefabce");
\end{asy}
\quad\raisebox{3.2cm}{$=$}
\begin{asy}
pen r = gray(0.7);
filldraw(myt * shift(6.2, 13.6) * inverse(myt) * scale(1.9,4) * unitcircle, r, white);
pen g = rgb(0.2, 1, 0.2);
filldraw(myt * shift(4.5, 8.5) * inverse(myt) * scale(3.5,5.5) * (((N..(0.5*(N+E))..E) & (E..(0.5*(E+S))..S) & (S..(0.5*(S+W))..W) & (W..(0.5*(W+N))..N))--cycle), g, white);
drawbraid("abcabadefgcdefbcdeabcdffgefabce");
\end{asy}
\quad\raisebox{3.2cm}{$=$}
\begin{asy}
pen r = gray(0.7);
pen q = rgb(0.9, 0.1, 0.1);
pen p = rgb(0.1, 0.1, 0.9);
filldraw(myt * shift(2.1, 9.4) * inverse(myt) * scale(1.9,4) * unitcircle, r, white);
filldraw(myt * shift(6.4, 2) * inverse(myt) * scale(2,1) * unitcircle, p, white);
filldraw(myt * shift(6.4, 18) * inverse(myt) * scale(2,1) * unitcircle, p, white);
pen g = rgb(0.2, 1, 0.2);
filldraw(myt * shift(4.5, 13.5) * inverse(myt) * scale(3.5,5.5) * (((N..(0.5*(N+E))..E) & (E..(0.5*(E+S))..S) & (S..(0.5*(S+W))..W) & (W..(0.5*(W+N))..N))--cycle), g, white);
drawbraid("abcababbacbadefgcdefbcdeabcdabc");
\end{asy}
\captionof{figure}{Braid isotopies of $\beta''$: the gray marked crossings get isotoped down through $\Diamond$ (green) according to~\eqref{eq:diamond}.}
\label{fig:beta''}
\end{figure}

Deleting and adding a braid generator changes the closure by smoothing a crossing and adding a crossing, respectively. Recalling that each of these correspond to a 1-handle, there is a genus one cobordism build out of two $1$-handles between the knots $T_{5,7}$ and the knot $K$ given as the closure of $\beta''$. To conclude $d(T_{4,9}, T_{5,7})=1$, we now show that $K=T_{4,9}$:

We rewrite the $8$-braid
\begin{align*}\beta''=&(a_1a_2a_3a_1a_2a_1)\Diamond (a_6
  a_6a_7a_5a_6a_5)(a_1a_2a_3)\\
\overset{\eqref{eq:diamond}}{=}&(a_1a_2a_3a_1a_2a_1)(a_2
  a_2a_3a_1a_2a_1)\Diamond (a_1a_2a_3)\end{align*}
and observe (see Figures~\ref{fig:beta''} and~\ref{fig:beta'''}) that it has the same closure as the $4$-braid $\beta'''$ given by the braid word obtained from the above by replacing $\Diamond$ by the so-called full-twist $\Delta^2=(a_1a_2a_3)^4$:
\[\beta'''\coloneqq(a_1a_2a_3a_1a_2a_1)(a_2
  a_2a_3a_1a_2a_1)\Delta^2 (a_1a_2a_3).\]
\begin{figure}[h]
\centering
\begin{asy}
pen r = gray(0.7);
filldraw(myt * shift(2.5, 6 + 1.5) * inverse(myt) * scale(0.85) * (((N..(0.4*(N+E))..E) & (E..(0.4*(E+S))..S) & (S..(0.4*(S+W))..W) & (W..(0.4*(W+N))..N))--cycle), r, white);
drawbraid("abcababbacbaabcabcabcabcabc");
\end{asy}
\quad\raisebox{3.7cm}{$=$}
\begin{asy}
pen r = gray(0.7);
filldraw(myt * shift(2.5, 1 + 1.5) * inverse(myt) * scale(0.85) * (((N..(0.4*(N+E))..E) & (E..(0.4*(E+S))..S) & (S..(0.4*(S+W))..W) & (W..(0.4*(W+N))..N))--cycle), r, white);
drawbraid("babcababacbaabcabcabcabcabc");
\end{asy}
\quad\raisebox{3.7cm}{$\overset{\text{cyc.~per.}}{\longleftrightarrow}$}
\begin{asy}
pen r = gray(0.7);
filldraw(myt * shift(2.5, 21 + 1.5) * inverse(myt) * scale(0.85) * (((N..(0.4*(N+E))..E) & (E..(0.4*(E+S))..S) & (S..(0.4*(S+W))..W) & (W..(0.4*(W+N))..N))--cycle), r, white);
drawbraid("abcababacbaabcabcabcabcabcb");
\end{asy}
\quad\raisebox{3.7cm}{$=$}
\begin{asy}
pen r = gray(0.7);
filldraw(myt * shift(1.5, 9 + 1.5) * inverse(myt) * scale(0.8) * unitcircle, r, white);
filldraw(myt * shift(3.5, 9 + 1.5) * inverse(myt) * scale(0.85) * (((N..(0.4*(N+E))..E) & (E..(0.4*(E+S))..S) & (S..(0.4*(S+W))..W) & (W..(0.4*(W+N))..N))--cycle), r, white);
drawbraid("abcababacbacabcabcabcabcabc");
\end{asy}
\quad\raisebox{3.7cm}{$=$}
\begin{asy}
pen r = gray(0.7);
filldraw(myt * shift(3.5, 5 + 1.5) * inverse(myt) * scale(0.8) * unitcircle, r, white);
drawbraid("abcabcabacbcabcabcabcabcabc");
\end{asy}
\captionof{figure}{Left: The $4$-braid $\beta'''$ with the same closure as $\beta''$. This is geometrically seen by closing up top and bottom of the 4 right most strands of $\beta''$ (i.e.~glue up the two blue half-ellipses in Figure~\ref{fig:beta''}).\\
Right: Cyclic permutation on $\beta'''$ (the gray marked crossing is removed on the bottom and added to the top) yields a braid which is $(a_1a_2a_3)^9$ (most right) up to braid isotopy (the braid isotopy is for the last two equalities are given by isotoping the gray marked crossing down).}
\label{fig:beta'''}
\end{figure} 
Further applications of braid relations (see most-left equality in Figure~\ref{fig:beta'''}) yield
\[\beta'''=a_2(a_1a_2a_3a_1a_2a_1)(a_2a_3a_1a_2a_1)\Delta^2 (a_1a_2a_3),\] and a cyclic permutation (see left-right arrow in Figure~\ref{fig:beta'''}) gives
\begin{align*}(a_1a_2a_3a_1a_2a_1)(a_2a_3a_1a_2a_1)\Delta^2 (a_1a_2a_3)a_2=&(a_1a_2a_3a_1a_2a_1)(a_2a_3a_1a_2a_1)a_3\Delta^2 (a_1a_2a_3)\\
=&(a_1a_2a_3a_1a_2a_1)a_3(a_2a_3a_1a_2)a_3\Delta^2 (a_1a_2a_3)\\
=&(a_1a_2a_3a_1a_2a_3a_1(a_2a_3a_1a_2)a_3\Delta^2 (a_1a_2a_3)\\
=&(a_1a_2a_3)(a_1a_2a_3)(a_1a_2a_3)(a_1a_2a_3)\Delta^2 (a_1a_2a_3)\\
=&(a_1a_2a_3)^9,
\end{align*}
which establishes that $\beta'''$ has closure $T(4,9)$.

\begin{figure}[h]
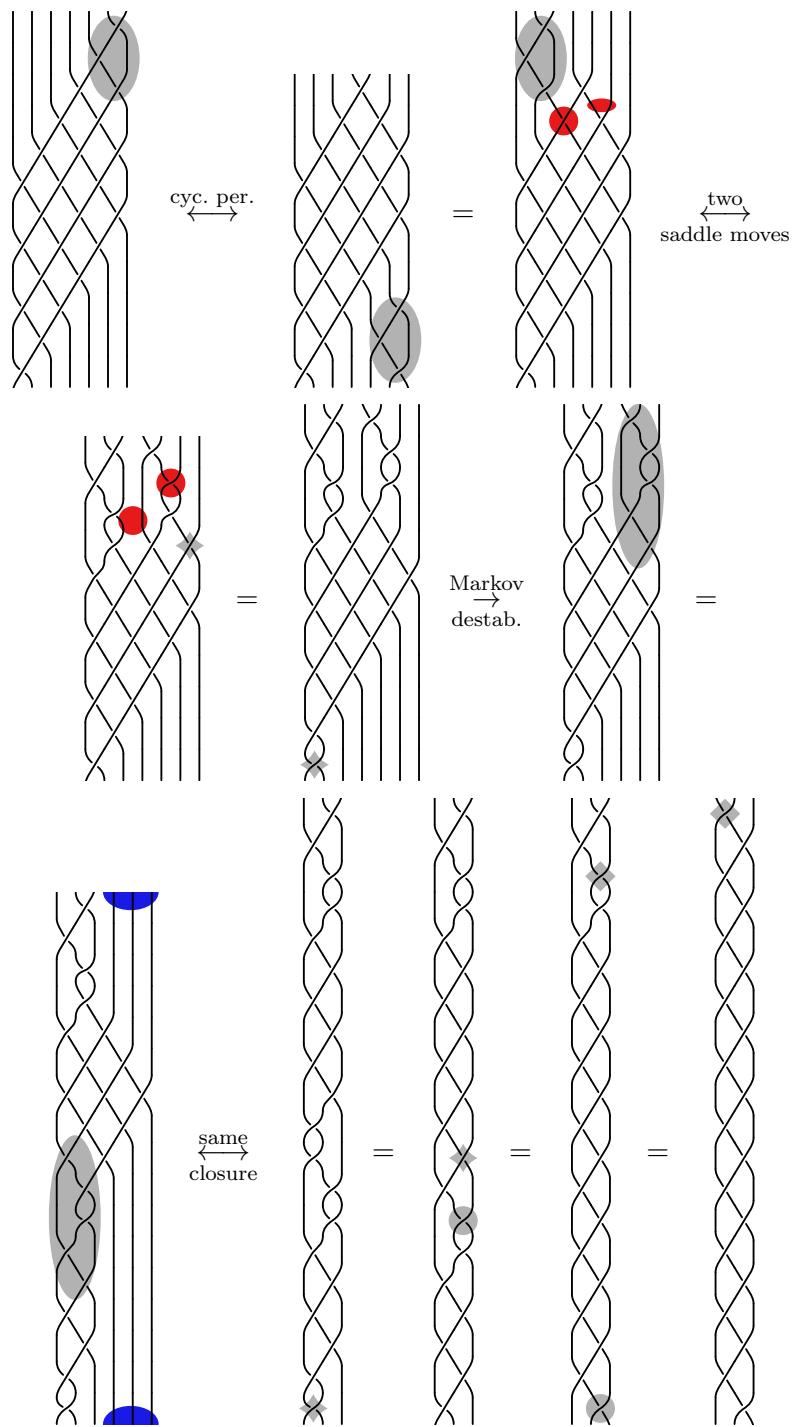

\centering
\begin{asy}
pen r = gray(0.7);
filldraw(myt * shift(6.3, 12.5) * inverse(myt) * scale(1.4,2.3) * unitcircle, r, white);
drawbraid("abcdefabcdefabcdefabcdef");
\end{asy}
\quad\raisebox{2.2cm}{$\overset{\text{cyc.~per.}}{\longleftrightarrow}$}
\begin{asy}
pen r = gray(0.7);
filldraw(myt * shift(6.3, 3.5) * inverse(myt) * scale(1.4,2.3) * unitcircle, r, white);
drawbraid("fefabcdefabcdefabcdeabcd");
\end{asy}
\quad\raisebox{2.2cm}{$=$}
\begin{asy}
pen r = gray(0.7);
pen q = rgb(0.9, 0.1, 0.1);
filldraw(myt * shift(2.3, 12.5) * inverse(myt) * scale(1.4,2.3) * unitcircle, r, white);
filldraw(myt * shift(3.5, 9 + 1.5) * inverse(myt) * scale(0.8) * unitcircle, q, white);
filldraw(myt * shift(5.5, 9.5 + 1.5) * inverse(myt) * scale(0.8,0.4) * unitcircle, q, white);

drawbraid("abcdefabcdefabcdeabcdbab");
\end{asy}
\quad\raisebox{2.2cm}{$\overset{\text{two}}{\underset{\text{saddle moves}}{\longleftrightarrow}}$}

\vspace{0.2cm}

\begin{asy}
pen r = gray(0.7);
pen q = rgb(0.9, 0.1, 0.1);
filldraw(myt * shift(3.5, 8.8 + 1.5) * inverse(myt) * scale(0.8) * unitcircle, q, white);
filldraw(myt * shift(5.5, 10 + 1.5) * inverse(myt) * scale(0.8) * unitcircle, q, white);
filldraw(myt * shift(6.5, 8 + 1.5) * inverse(myt) * scale(0.85) * (((N..(0.4*(N+E))..E) & (E..(0.4*(E+S))..S) & (S..(0.4*(S+W))..W) & (W..(0.4*(W+N))..N))--cycle), r, white);
drawbraid("abcdefabcdefabcdeabedbab");
\end{asy}
\quad\raisebox{2.3cm}{$=$}
\begin{asy}
pen r = gray(0.7);
pen q = rgb(0.9, 0.1, 0.1);
filldraw(myt * shift(1.5, 1 + 1.5) * inverse(myt) * scale(0.85) * (((N..(0.4*(N+E))..E) & (E..(0.4*(E+S))..S) & (S..(0.4*(S+W))..W) & (W..(0.4*(W+N))..N))--cycle), r, white);
drawbraid("aabcdefabcdeabcdeabedbab");
\end{asy}
\quad\raisebox{2.3cm}{$\underset{\text{destab.}}{\overset{\text{Markov}}{\to}}$}
\begin{asy}
pen r = gray(0.7);
pen q = rgb(0.9, 0.1, 0.1);
filldraw(myt * shift(4.9, 11.4) * inverse(myt) * scale(1.4,4.4) * unitcircle, r, white);
drawbraid("aabcdeabcdeabcdeabedbab");
\end{asy}
\quad\raisebox{2.3cm}{$=$}

\vspace{0.2cm}

\begin{asy}
pen r = gray(0.7);
pen q = rgb(0.9, 0.1, 0.1);
filldraw(myt * shift(1.95, 8.6) * inverse(myt) * scale(1.4,4.4) * unitcircle, r, white);
pen p = rgb(0.1, 0.1, 0.9);
filldraw(myt * shift(4.9, 2) * inverse(myt) * scale(1.5,1) * unitcircle, p, white);
filldraw(myt * shift(4.9, 19) * inverse(myt) * scale(1.5,1) * unitcircle, p, white);
drawbraid("aabababbacdebcdabcabbab");
\end{asy}
\quad\raisebox{3.5cm}{$\underset{\text{closure}}{\overset{\text{same}}{\longleftrightarrow}}$}
\begin{asy}
pen r = gray(0.7);
filldraw(myt * shift(1.5, 1 + 1.5) * inverse(myt) * scale(0.85) * (((N..(0.4*(N+E))..E) & (E..(0.4*(E+S))..S) & (S..(0.4*(S+W))..W) & (W..(0.4*(W+N))..N))--cycle), r, white);
drawbraid("aabababbaababababbab");
\end{asy}
\quad\raisebox{3.5cm}{$=$}
\begin{asy}
pen r = gray(0.7);
filldraw(myt * shift(2.5, 9 + 1.5) * inverse(myt) * scale(0.85) * (((N..(0.4*(N+E))..E) & (E..(0.4*(E+S))..S) & (S..(0.4*(S+W))..W) & (W..(0.4*(W+N))..N))--cycle), r, white);
filldraw(myt * shift(2.5, 7 + 1.5) * inverse(myt) * scale(0.8) * unitcircle, r, white);
drawbraid("abababbabababababbab");
\end{asy}
\quad\raisebox{3.5cm}{$=$}
\begin{asy}
pen r = gray(0.7);
filldraw(myt * shift(2.5, 1 + 1.5) * inverse(myt) * scale(0.8) * unitcircle, r, white);
filldraw(myt * shift(2.5, 18 + 1.5) * inverse(myt) * scale(0.85) * (((N..(0.5*(N+E))..E) & (E..(0.5*(E+S))..S) & (S..(0.5*(S+W))..W) & (W..(0.5*(W+N))..N))--cycle), r, white);
drawbraid("bababababababababbab");
\end{asy}
\quad\raisebox{3.5cm}{$=$}
\begin{asy}
pen r = gray(0.7);
filldraw(myt * shift(1.5, 20 + 1.5) * inverse(myt) * scale(0.85) * (((N..(0.5*(N+E))..E) & (E..(0.5*(E+S))..S) & (S..(0.5*(S+W))..W) & (W..(0.5*(W+N))..N))--cycle), r, white);
drawbraid("babababababababababa");
\end{asy}

\captionof{figure}{The standard 7-braid with closure $T_{4,7}$ (top left) is cyclically permuted. Then a generator is added and one is removed, which corresponds to smoothing a crossing and and adding a crossing (top right to middle left, red). The resulting braid can be Markov destabilized (preserves closure) and has the same closure as a 3-braid that has $T_{3,10}$ as its closure.}
\label{fig:47to310}
\end{figure}

\noindent\textbf{Construction of a genus one cobordism between $T_{3,10}$ and $T_{4,7}$:}
Again by using braid relations changing between braids with the same closure and adding one and deleting one generator,
we manage to get from a braid with closure $T_{4,7}$ to a braid with closure $T_{3,10}$. The argument is very similar in style to the construction of the cobordism between $T_{4,9}$ and $T_{5,7}$ above. Rather than providing the relevant braid words, we provide the corresponding diagrammatic proof; see Figure~\ref{fig:47to310}. As before, this shows the existence of a genus one cobordism between $T_{3,10}$ and $T_{4,7}$ given by two 1-handles.

\begin{figure}[h]
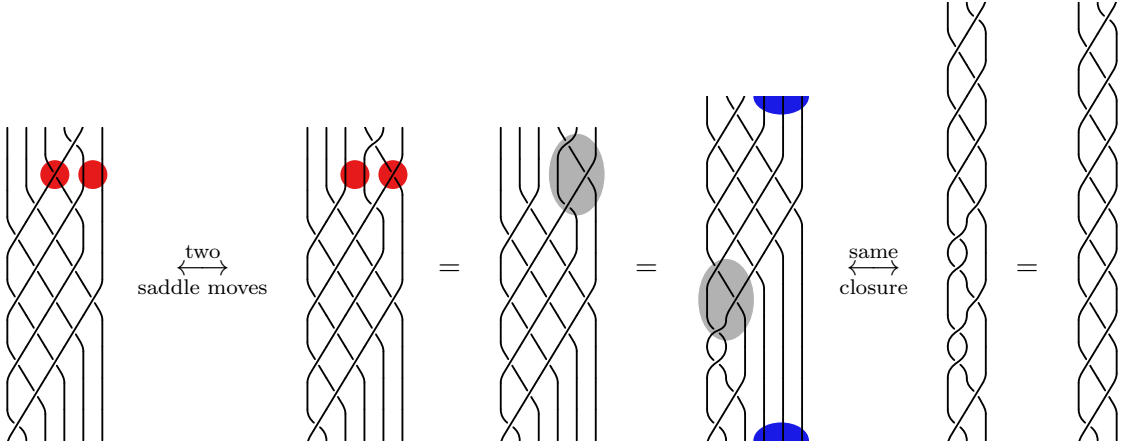

\centering
\begin{asy}
pen q = rgb(0.9, 0.1, 0.1);
filldraw(myt * shift(3.5, 9 + 1.5) * inverse(myt) * scale(0.8) * unitcircle, q, white);
filldraw(myt * shift(5.5, 9 + 1.5) * inverse(myt) * scale(0.8) * unitcircle, q, white);
drawbraid("abcdeabcdabcdabcd");
\end{asy}
\quad\raisebox{2.2cm}{$\overset{\text{two}}{\underset{\text{saddle moves}}{\longleftrightarrow}}$}
\begin{asy}
pen q = rgb(0.9, 0.1, 0.1);
filldraw(myt * shift(3.5, 9 + 1.5) * inverse(myt) * scale(0.8) * unitcircle, q, white);
filldraw(myt * shift(5.5, 9 + 1.5) * inverse(myt) * scale(0.8) * unitcircle, q, white);
drawbraid("abcdeabcdabcdeabd");
\end{asy}
\quad\raisebox{2.2cm}{$=$}
\begin{asy}
pen r = gray(0.7);
filldraw(myt * shift(5, 10.5) * inverse(myt) * scale(1.5,2.2) * unitcircle, r, white);
drawbraid("abcdeabcdabcdeabd");
\end{asy}
\quad\raisebox{2.2cm}{$=$}
\begin{asy}
pen r = gray(0.7);
filldraw(myt * shift(2, 6.5) * inverse(myt) * scale(1.5,2.2) * unitcircle, r, white);
pen p = rgb(0.1, 0.1, 0.9);
filldraw(myt * shift(4.9, 2) * inverse(myt) * scale(1.5,1) * unitcircle, p, white);
filldraw(myt * shift(4.9, 13) * inverse(myt) * scale(1.5,1) * unitcircle, p, white);
drawbraid("abaabacdebcdabcab");
\end{asy}
\quad\raisebox{2.2cm}{$\underset{\text{closure}}{\overset{\text{same}}{\longleftrightarrow}}$}
\begin{asy}
drawbraid("abaabaabababab");
\end{asy}
\quad\raisebox{2.2cm}{$=$}
\begin{asy}
drawbraid("ababababababab");
\end{asy}
\captionof{figure}{A $6$-braid with closure $T_{4,5}$ (left) that can be modified by adding and removing one crossing to yield a braid with closure $T_{3,7}$.}
\label{fig:45to37}
\end{figure}

\noindent\textbf{Construction of a genus one cobordism between $T_{4,5}$ and $T_{3,7}$:}
Once more by using braid relations changing between braids with the same closure and adding one and deleting one generator,
we manage to get from a braid with closure $T_{4,5}$ to a braid with closure $T_{3,7}$; see Figure~\ref{fig:45to37}.
As before, this yields a genus one cobordism from $T_{3,7}$ and $T_{4,5}$ given by two 1-handles.\end{proof}

We end this section with a remark on the cobordisms from Proposition~\ref{prop:existenceofcobs}.

\begin{remark}\label{rmk:cobordisms}From Remark~\ref{rem:abc} and the explicit depiction of the cobordisms in the proof of Proposition~\ref{prop:existenceofcobs}, we conclude the following. All cobordisms constructed in the proof of Proposition~\ref{prop:existenceofcobs} are obtained as a composition of the following moves on positive braids:
\begin{itemize}
\item switching between positive braids with the same closure (an operation which can be achieved by positive Markov stabilization and destabilization of positive braids; see~\cite[Corollary~1.13]{Etnyre-VanHornMorris:2011-1}),
\item positive Markov stabilization or destabilization,
\item cyclic permutation of a positive braid word, and
\item removal or addition of a positive generator to a braid word representing a given positive braid.
\end{itemize}
In fact, the genus one cobordisms for the pairs appearing in (1)--(5) and $\{T_{2,11}, T_{4,5}\}$ are obtained by removing two generators, and the genus one cobordisms between the remaining five pairs of (6) are obtained by removing one generator and adding one.\end{remark}

%%%%%%%%%%%%%%%%%%%%%%%%%%%%%%%%%%%%%%%%%%
\section{Proof of Theorem~\ref{thm:main} Part 2: Obstructing cobordisms 
}\label{sec:proofThm1.1}

First, we record some useful properties of the Ozsv{\'a}th-Szab{\'o} $\tau$ invariant \cite{Ozsvath-Szabo:2003-3} and the Hom-Wu $\nu^+$ invariant \cite{Hom-Wu:2016-1}. Recall that the $\tau$ invariant is a group homomorphism from the concordance group to $\mathbb{Z}$ and the $\nu^+$ invariant is non-negative integer valued concordance invariant. Further, the $\tau$ invariant detects the $4$-ball genus of torus knots.

\begin{proposition}[{\cite[Corollary 1.7]{Ozsvath-Szabo:2003-3}}]\label{prop:tau-torus} For any positive torus knot $T_{p,q}$, we have
\[\pushQED{\qed}\tau(T_{p,q}) = g_4(T_{p,q}) = \frac{(p-1)(q-1)}{2}. \qedhere\]
\end{proposition}

\noindent Further, there is an inequality between $\tau$ and $\nu^+$.

\begin{proposition}[{\cite[Proposition 2.3 and Proposition 2.4]{Hom-Wu:2016-1}}]\label{prop:tau-nu-inequal} For any knot $K$,
\[\pushQED{\qed} \tau(K) \leq  \nu^+(K) \leq g_4(K). \qedhere\]
\end{proposition}

Every positive torus knot $T_{p,q}$ has an associated \emph{semigroup} $\Gamma_{p,q}:= \langle p,q\rangle\subset\N$ (the subsemigroup of $\N$ generated by $p$ and $q$). The $n$th element of $\Gamma_{p,q}$ is denoted by $\Gamma_{p,q}(n)$, and $\displaystyle\max\{\Gamma_{p,q}(n) - \Gamma_{p',q'}(n) \mid n\geq 1\}$ is denoted by $\Gamma_{p,q:p',q'}$. We note that $\Gamma_{p,q:p',q'}$ is always finite since $\Gamma_{p,q}$ contains all natural numbers greater than $pq-p-q$~\cite{Sylvester:1882-1}. The $\nu^+$ invariant of the difference of two positive torus knots can be computed as follows.

\begin{proposition}[{\cite[Theorem 1.1]{Bodnar-Celoria-Golla:2017-1}}]\label{prop:BCG}  {For any positive torus knots  $T_{p,q}$ and $T_{p',q'}$,} we have
\[\pushQED{\qed}\nu^+(T_{p,q} \# -T_{p',q'}) = \max\{\tau(T_{p,q} \# -T_{p',q'}) + \Gamma_{p',q':p,q},0\}.\qedhere\]
\end{proposition}

We first briefly check the simpler direction of Proposition~\ref{prop:nu}; that is that the pairs of knots listed, indeed satisfy \[ \max \{ \nu^+(T_{p,q} \# -T_{p',q'}), \nu^+(T_{p',q'} \# -T_{p,q}) \} \leq 1.\]
In fact, we also calculate the minimum (rather than the maximum) of $\nu^+(T_{p,q} \# -T_{p',q'})$ and $\nu^+(T_{p',q'} \# -T_{p,q})$, which we will use in Section~\ref{sec:Gordian}.

\begin{lemma}\label{lem:minandmaxnu^+}
For all the pairs of knots $\eqref{fam1}$--$\eqref{fam7}$ in Theorem~\ref{thm:main}, we have
\[ \max \{ \nu^+(T_{p,q} \# -T_{p',q'}), \nu^+(T_{p',q'} \# -T_{p,q}) \} = 1.\]
\noindent Furthermore, we have
\[ \min \{ \nu^+(T_{p,q} \# -T_{p',q'}), \nu^+(T_{p',q'} \# -T_{p,q}) \} = 1,\]
for the pairs $\eqref{fam3}$, $\eqref{fam4}$, $\eqref{fam5}$ if $n\in\{7,8\}$, the latter $4$ pairs of $\eqref{fam6}$, and $\eqref{fam7}$. In contrast, the remaining pairs provided in Theorem~\ref{thm:main} satisfy
\[ \min \{ \nu^+(T_{p,q} \# -T_{p',q'}), \nu^+(T_{p',q'} \# -T_{p,q}) \} = 0.\]
\end{lemma}
\begin{proof} For the pairs $\eqref{fam1}$, we have $\Gamma_{2,2n+1:2,2n+3}=0$ and $\Gamma_{2,2n+3:2,2n+1}=1$. Using Proposition~\ref{prop:BCG}, we get $\nu^+(T_{2,2n+1} \# -T_{2,2n+3})=0$ and $\nu^+(T_{2,2n+3} \# -T_{2,2n+1})=1$. The same argument applies for the pairs $\eqref{fam2}$.

For the pairs $\eqref{fam3}$, by Proposition~\ref{prop:existenceofcobs}, \ref{prop:tau-torus}, and \ref{prop:tau-nu-inequal} we have 
\[ \nu^+(T_{3n+1,9n+6} \# -T_{3n+2,9n+3})\leq 1\ \text{ and }\ \nu^+(T_{3n+2,9n+3} \# -T_{3n+1,9n+6})=1.\] Further, an easy computation gives $\Gamma_{3n+1,9n+6}(3)=6n+2$ and $\Gamma_{3n+2,9n+3}(3)=6n+4$. Therefore, by Proposition~\ref{prop:BCG}, we have
\begin{equation*}\begin{split}
1 = -1+ \Gamma_{3n+2,9n+3}(3)- \Gamma_{3n+1,9n+6}(3) &\leq -1 +\Gamma_{3n+2,9n+3:3n+1,9n+6}\\
& \leq \nu^+(T_{3n+1,9n+6} \# -T_{3n+2,9n+3}).
\end{split}
\end{equation*} Hence $\nu^+(T_{3n+1,9n+6} \# -T_{3n+2,9n+3})=1$. The same argument applies for the pairs $\eqref{fam4}$. Lastly, the rest of the cases easily follow from Proposition~\ref{prop:BCG}.
\end{proof}

\noindent Now, we are ready to prove Proposition~\ref{prop:nu}.

\begin{proof}[Proof of Proposition~\ref{prop:nu}]
By Lemma~\ref{lem:minandmaxnu^+}, it remains to show that, whenever
\[\max \{ \nu^+(T_{p,q} \# -T_{p',q'}), \nu^+(T_{p',q'} \# -T_{p,q}) \} \leq 1,\]
then $\{T_{p,q}, T_{p',q'} \}$ must be one of the pairs listed in $\eqref{fam1}$--$\eqref{fam7}$.

Suppose $ \max \{ \nu^+(T_{p,q} \# -T_{p',q'}), \nu^+(T_{p',q'} \# -T_{p,q}) \} \leq 1.$ Moreover, w.l.o.g. we may assume that $p \leq p'$, $p<q$, and $p'<q'$. By Proposition~\ref{prop:tau-nu-inequal} we have $|\tau(T_{p,q} \# - T_{p',q'})| \leq 1$ and by Proposition~\ref{prop:tau-torus} we have
\[(p-1)(q-1)-(p'-1)(q'-1) \in \{ -2, 0 ,2\}.\] We deal with each case separately.

\textbf{Case 1: \begin{equation}\label{eqn:case1}(p-1)(q-1)-(p'-1)(q'-1)=-2.
\end{equation}}\noindent Combining $\nu^+(T_{p,q} \# -T_{p',q'}) \leq 1$ and Proposition~\ref{prop:BCG}, we have
\[p'-p = \Gamma_{p',q'}(2)-\Gamma_{p,q}(2) \leq \Gamma_{p',q':p,q} \leq \nu^+(T_{p,q} \# -T_{p',q'})+1 \leq 2.\]
Similarly, since $\nu^+(T_{p',q'} \# -T_{p,q}) \leq 1$, we have
\[p-p' = \Gamma_{p,q}(2)-\Gamma_{p',q'}(2) \leq \Gamma_{p,q:p',q'} \leq \nu^+(T_{p',q'} \# -T_{p,q})-1 \leq 0.\]
Combining above two inequalities, we have $0\leq p'-p\leq 2$.

If $p'-p=0$, then the equation \eqref{eqn:case1} simplifies to $(p-1)(q'-q)=2$. Then either $p=2$ and $q'-q=2$ or $p=3$ and $q'-q=1$. These are listed in $\eqref{fam1}$ and $\eqref{fam2}$.

If $p'-p=1$, then the equation \eqref{eqn:case1} simplifies to $q=q'+\frac{q'-3}{p-1}$. Let $m=\frac{q'-3}{p-1}$, so that $q=q'+m$. Note that $m$ is a positive integer, since $q'>p'>p \geq 2$. We can rewrite the pairs as $\{T_{p,mp+3}, T_{p+1,m(p-1)+3}\}$. First, suppose $m=1$, namely, we are considering the pairs $\{T_{p,p+3}, T_{p+1,p+2}\}$. If $p \geq 3$, then $\Gamma_{p,p+3}(3) = p+3$ and $\Gamma_{p+1,p+2}(3) = p+2$. Further, we have
\[1 = \Gamma_{p,p+3}(3)- \Gamma_{p+1,p+2}(3) \leq \Gamma_{p,p+3:p+1,p+2} \leq \nu^+(T_{p+1,p+2} \# -T_{p,p+3})-1 \leq 0,\]
which is a contradiction. If $p=2$, we get $\{T_{2,5}, T_{3,4}\}$ which is listed in $\eqref{fam5}$. When $m=2,4,$ and $5$, a similar argument gives $\{T_{2,7}, T_{3,5}\}, \{T_{2,11}, T_{3,7}\},$ and $\{T_{2,13}, T_{3,8}\}$ which are listed in $\eqref{fam5}$. When $m=3$, the pairs are listed in $\eqref{fam3}$. Now, suppose that $m\geq 6$, then $\Gamma_{p,mp+3}(4)=3p$ and $\Gamma_{p+1,m(p-1)+3}(4)=3p+3$. Therefore we have
\begin{equation*}\begin{split}
3 = \Gamma_{p+1,m(p-1)+3}(4)- \Gamma_{p,mp+3}(4) &\leq \Gamma_{p+1,m(p-1)+3: p,mp+3}\\
& \leq \nu^+(T_{p,mp+3} \# -T_{p+1,m(p-1)+3})+1\\
&\leq 2,
\end{split}
\end{equation*}which is a contradiction.

If $p'-p=2$, then the equation \eqref{eqn:case1} simplifies to $q=q'+\frac{2q'-4}{p-1}$. Let $m=2\cdot\frac{q'-2}{p-1}$, so that $q=q'+m$. Note that since $q'>p'> p \geq 2$, we see that $m$ is a positive integer greater than $2$. We can rewrite the pairs as $\{T_{p,\frac{m(p+1)+4}{2}}, T_{p+2,\frac{m(p-1)+4}{2}}\}$. First, suppose $m=3$, then the pairs are $\{T_{p,\frac{3p+7}{2}}, T_{p+2,\frac{3p+1}{2}}\}$ and note that $p$ has to be odd. If $3 \leq p \leq 7$, then $\Gamma_{p,\frac{3p+7}{2}}(3) = 2p$ and $\Gamma_{p+2,\frac{3p+1}{2}}(3) = \frac{3p+1}{2}$. Then we have
\begin{equation*}\begin{split}
2p-\frac{3p+1}{2}=\Gamma_{p,\frac{3p+7}{2}}(3)-\Gamma_{p+2,\frac{3p+1}{2}}(3)
&\leq \Gamma_{p,\frac{3p+7}{2}:p+2,\frac{3p+1}{2}}\\
&\leq \nu^+(T_{p+2,\frac{3p+1}{2}} \# -T_{p,\frac{3p+7}{2}})-1\\
&\leq 0,
\end{split}
\end{equation*}which is a contradiction. If $p>7$, then $\Gamma_{p,\frac{3p+7}{2}}(3) = \frac{3p+7}{2}$ and $\Gamma_{p+2,\frac{3p+1}{2}}(3) = \frac{3p+1}{2}$. Similarly as above, it is straight forward to check that this leads to a contradiction. For the case $m=4$, we get the pairs that are listed in $\eqref{fam4}$. Now, assume $m>4$. When $1+\frac{8}{m-4} \leq p$, we have $\Gamma_{p,\frac{m(p+1)+4}{2}}(3) = 2p$ and $\Gamma_{p+2,\frac{m(p-1)+4}{2}}(3) = 2p+4$. We have
\begin{equation*}\begin{split}
4 = \Gamma_{p+2,\frac{m(p-1)+4}{2}}(3)- \Gamma_{p,\frac{m(p+1)+4}{2}}(3) &\leq \Gamma_{p+2,\frac{m(p-1)+4}{2}:p,\frac{m(p+1)+4}{2}}\\
& \leq \nu^+(T_{p,\frac{m(p+1)+4}{2}} \# -T_{p+2,\frac{m(p-1)+4}{2}})+1\\
&\leq 2,
\end{split}
\end{equation*} which is a contradiction. When $1+\frac{6}{m-4} \leq p < 1+\frac{8}{m-4}$, we have $\Gamma_{p,\frac{m(p+1)+4}{2}}(3) = 2p$ and $\Gamma_{p+2,\frac{m(p-1)+4}{2}}(3) = \frac{m(p-1)+4}{2}$. It is straight forward to check that a similar argument as above gives a contradiction. Hence we only need to consider the case when $p<1+\frac{6}{m-4}$. Since $p\geq 2$, we see that $m\leq 9$. There are three possible pairs, $\{T_{5,17}, T_{7,12}\}, \{T_{2,11}, T_{4,5}\},$ and $\{T_{3,14}, T_{5,8}\}$. The first pair can be ruled out, since $\nu^+(T_{7,12}\# - T_{5,17})=2$, the second pair is listed in $\eqref{fam6}$, and the last pair is listed in $\eqref{fam7}$.

\textbf{Case 2: \begin{equation}\label{eqn:case2}(p-1)(q-1)-(p'-1)(q'-1)=0.\end{equation}}\noindent Combining $\nu^+(T_{p',q'} \# -T_{p,q}) \leq 1$ and Proposition~\ref{prop:BCG}, we have
\[p'-p = \Gamma_{p',q'}(2)-\Gamma_{p,q}(2) \leq \Gamma_{p',q':p,q} \leq \nu^+(T_{p,q} \# -T_{p',q'}) \leq 1.\] Since we are assuming that $p\leq p'$, we have $p'-1 \leq p \leq p'$. Further, $p \neq p'$, since we are assuming that we have a pair of two distinct torus knots. Hence $p=p'-1$ and the equation \eqref{eqn:case2} simplifies to $q=q'+\frac{q'-1}{p-1}$. Let $m=\frac{q'-1}{p-1}$, so that $q=q'+m$. Since $q'>p'>p \geq 2$, we see that $m$ is a positive integer greater than $1$. The pairs can be rewritten as $\{T_{p,mp+1},T_{p+1,m(p-1)+1} \}$. If $m=2$, the pairs are $\{T_{p,2p+1},T_{p+1,2p-1} \}$. If $p\geq 5$, then $\Gamma_{p,2p+1}(6) = 3p+1$ and $\Gamma_{p+1,2p-1}(6) = 3p+3$. Therefore, we have
\begin{equation*}\begin{split}
2 = \Gamma_{p+1,2p-1}(6)- \Gamma_{p,2p+1}(6) &\leq \Gamma_{p+1,2p-1:p,2p+1}\\
& \leq \nu^+(T_{p,2p+1} \# -T_{p+1,2p-1})\\
&\leq 1,
\end{split}
\end{equation*}
which is a contradiction. When $3\leq p \leq 5$, we get pairs $\{T_{3,7}, T_{4,5}\}$ and $\{T_{4,9}, T_{5,7}\}$ which are listed in $\eqref{fam6}$. If $m\geq 3$ and $p \geq 1+\frac{3}{m-2}$, then $\Gamma_{p,mp+1}(3) = 2p$ and $\Gamma_{p+1,m(p-1)+1}(3) = 2p+2$. Therefore, we have
\begin{equation*}\begin{split}
2 = \Gamma_{p+1,m(p-1)+1}(3)- \Gamma_{p,mp+1}(3) &\leq \Gamma_{p+1,m(p-1)+1:p,mp+1}\\
& \leq \nu^+(T_{p,mp+1} \# -T_{p+1,m(p-1)+1})\\
&\leq 1,
\end{split}
\end{equation*}
which is a contradiction. The remaining cases are when $m\geq 3$ and $2 \leq p < 1+\frac{3}{m-2}$, and we get pairs $\{T_{2,7}, T_{3,4}\}, \{T_{3,10}, T_{4,7}\},$ and $\{T_{2,9}, T_{3,5}\}$ which are listed in $\eqref{fam6}$.

\textbf{Case 3: \begin{equation}\label{eqn:case3}(p-1)(q-1)-(p'-1)(q'-1)=2.\end{equation}}\noindent Combining $\nu^+(T_{p,q} \# -T_{p',q'}) \leq 1$ and Proposition~\ref{prop:BCG}, we have
\[p'-p = \Gamma_{p',q'}(2)-\Gamma_{p,q}(2) \leq \Gamma_{p',q':p,q} \leq \nu^+(T_{p,q} \# -T_{p',q'}) -1 \leq 0.\] Since we are assuming that $p\leq p'$, we have $p=p'$. The equation \eqref{eqn:case3} simplifies to $(p-1)(q-q')=2$. Then either $p=2$ and $q-q'=2$ or $p=3$ and $q-q'=1$. These are listed in $\eqref{fam1}$ and $\eqref{fam2}$.\end{proof}

Given that we have now established Proposition~\ref{prop:nu}, we take it together with the construction of cobordisms in the previous section to conclude Theorem~\ref{thm:main}.
\begin{proof}[Proof of Theorem~\ref{thm:main}]
Proposition~\ref{prop:existenceofcobs} yields a genus one cobordism, whenever the pair of torus knots is in one of the families \eqref{fam1}--\eqref{fam6}.

The in particular part of Proposition~\ref{prop:nu} establishes that, whenever the pair of torus knots is in one of the families \eqref{fam1}--\eqref{fam7}, then the cobordism distance is at least two.
\end{proof}
%%%%%%%%%%%%%%%%%%%%%%%%%%%%%%%%%%%%%%%%%%
\section{Proof of Corollary~\ref{cor:3}: Obstructing crossing changes}\label{sec:Gordian}
We first recall the behavior of the $\nu^+$ invariant with respect to crossing changes.

\begin{proposition}[{\cite[Theorem 1.3]{Bodnar-Celoria-Golla:2017-1}}]\label{prop:nucrossing} If $K_+$ is obtained from $K_-$ by changing a negative crossing into a positive one, then
\[\pushQED{\qed}\nu^{+}(K_{+})-1\leq \nu^{+}(K_-)\leq \nu^{+}(K_{+}).\qedhere\]
\end{proposition}

Recall that for a given knot $K$ and a unit complex number $\omega$, Tristram and Levine  defined the $\omega$-signature
$\sigma_\omega(K)$ to be the signature of the hermitian matrix
\[M_\omega = (1-\omega)A + (1-\bar{\omega})A^{\top},\] where $A$ is a Seifert matrix of $K$ \cite{Levine:1969-1, Tristram:1969-1}. If $M_ \omega$ has non-zero determinant (equivalently, $\omega$ is not a root of the Alexander polynomial of $K$), $\omega$ is called \emph{regular} for $K$. There is an analogous property of Proposition~\ref{prop:nucrossing} for the Tristram-Levine signatures which follows from simple consideration of the Seifert
matrices (see e.g.~\cite{Giller:1982-1, Lipson:1990-1}).
\begin{proposition}\label{prop:signaturecrossing} If $K_+$ is obtained from $K_-$ by changing a negative crossing into a positive one, then, for $\omega\in S^1$ that are regular for $K_{+}$ and $K_{-}$, we have
\[\pushQED{\qed}\frac{-\sigma_\omega(K_{+})}{2}-1\leq \frac{-\sigma_\omega(K_{-})}{2}\leq \frac{-\sigma_\omega(K_{+})}{2}.\qedhere\]
\end{proposition}

\begin{proof}[Proof of Corollary~\ref{cor:3}]
The if-part is well-known to experts, and can be readily seen from the standard knot diagrams arising from positive braid closures. For an explicit reference, we point to the more general statements of Theorem~2 and Theorem~3 in~\cite{Feller_14_GordianAdjacency}.

For the only if direction, we use Proposition~\ref{prop:nucrossing}. We first note that, by Proposition~\ref{prop:nu}, all pairs different from $\eqref{fam1}$--$\eqref{fam7}$ in Theorem~\ref{thm:main} have
\[\max \{ \nu^+(T_{p,q} \# -T_{p',q'}), \nu^+(T_{p',q'} \# -T_{p,q}) \}\geq 2;\] thus, at least two crossing changes (in fact of the same type) are needed to turn $T_{p,q} \# -T_{p',q'}$ into a knot with $\nu^+=0$. In particular, for all pairs different from $\eqref{fam1}$--$\eqref{fam7}$, at least two crossing changes are needed to turn $T_{p,q}$ into $T_{p',q'}$ (or even a knot concordant to $T_{p',q'}$). Consequently, we are left to consider the pairs $\eqref{fam1}$--$\eqref{fam7}$.

Next, we note that the pairs described in the `furthermore'-part of Lemma~\ref{lem:minandmaxnu^+}, one needs at least 2 crossing changes to turn one of the knots into the other. Indeed, Proposition~\ref{prop:nucrossing} implies that knots $K$ and $L$ with $\nu^{+}(K\#-L)>0$ have the property that any sequence of crossing changes turning $K$ into $L$ must contain at least one positive-to-negative crossing change (since even any sequence of crossing changes that turn $K\#-L$ into a knot with $\nu^+=0$ must contain a positive-to-negative crossing change).
Consequently, knots $K$ and $L$ such that $\nu^{+}(K\#-L)>0$ and $\nu^{+}(L\#-K)>0$ have Gordian distance at least two.

It remains to discuss the pairs $\{ T_{2,7}, T_{3,4}\}$ and $\{ T_{2,9}, T_{3,5}\}$. We note that $\nu^+(T_{2,7}\#-T_{3,4})=\nu^+(T_{2,9}\#-T_{3,5})=1$, which can be easily verified by using Proposition~\ref{prop:BCG}. Therefore (as argued in the last paragraph) any sequence of crossing changes turning $T_{2,7}$ and $T_{2,9}$ into $T_{3,4}$ and $T_{3,5}$, respectively, must contain at least one positive-to-negative crossing change. We complete the proof by using Proposition~\ref{prop:signaturecrossing}. Indeed, a short calculation yields
\[\frac{-\sigma_{e^{2\pi i\alpha}}(T_{3,4}\#-T_{2,7})}{2}=\frac{-\sigma_{e^{2\pi i\beta}}(T_{3,5}\# -T_{2,9})}{2}=1\]
\[\text{for }\alpha\in \left(\frac{2}{12},\frac{3}{14}\right) \text{ and } \beta\in \left(\frac{2}{15},\frac{3}{18}\right)\cup \left(\frac{4}{15},\frac{5}{18}\right).\]
This calculation can, for example, be done using~\cite[Proposition~5.1]{GambaudoGhys_BraidsSignatures}, which is a generalization of the classical formula by Brieskorn and Hirzebruch~\cite{Brieskorn_DifftopovonSing,Hirzebruch_SingExoticBourbaki} for the classical signature $\sigma=\sigma_{-1}$ of torus knots. Therefore, any sequence of crossing changes turning $T_{2,7}$ and $T_{2,9}$ into $T_{3,4}$ and $T_{3,5}$, respectively, must contain at least one negative-to-positive crossing. Consequently, the pairs $\{ T_{2,7}, T_{3,4}\}$ and $\{ T_{2,9}, T_{3,5}\}$ have Gordian distance at least two.\end{proof}

\begin{remark}
 We note that for the only if direction of the proof for Corollary~\ref{cor:3} for the pairs $\{ T_{2,7}, T_{3,4}\}$ and $\{ T_{2,9}, T_{3,5}\}$, we could have used the Tristram-Levine signatures only, rather than $\nu^+$.

More generally, one may wonder whether Corollary~\ref{cor:3} can be obtained by use of only Tristram-Levine signatures as obstructions. This would in particular imply that the pairs listed in Corollary~\ref{cor:3} are the only pairs of torus knots that arise as the boundary of a properly immersed locally flat annulus in $S^3\times [0,1]$ that selfintersects transversely and in at most one point. However, the formula for the Tristram-Levine signature ends up being somewhat involved and the authors do not see how to use them to obtain the result.\end{remark}

%%%%%%%%%%%%%%%%%%%%%%%%%%%%%%%%%%%%%%%%%

\section{Positive braids and decomposable Lagrangian cobordisms}\label{sec:lag}
In this section, we prove Lemma~\ref{lemma:posbraidstodecompcobs} and Theorem~\ref{thm:genus1deccob}.
For this, we start with a brief overview of Lagrangian cobordism between Legendrian links. For more details, see e.g.~\cite{Etnyre:2005-1, Chantraine:2010-1, Cornwell-Ng-Sivek:2016-1, Baldwin-Lidman-Wong:2019-1}. 

\begin{figure}[htbp]
\includegraphics[width=12cm]{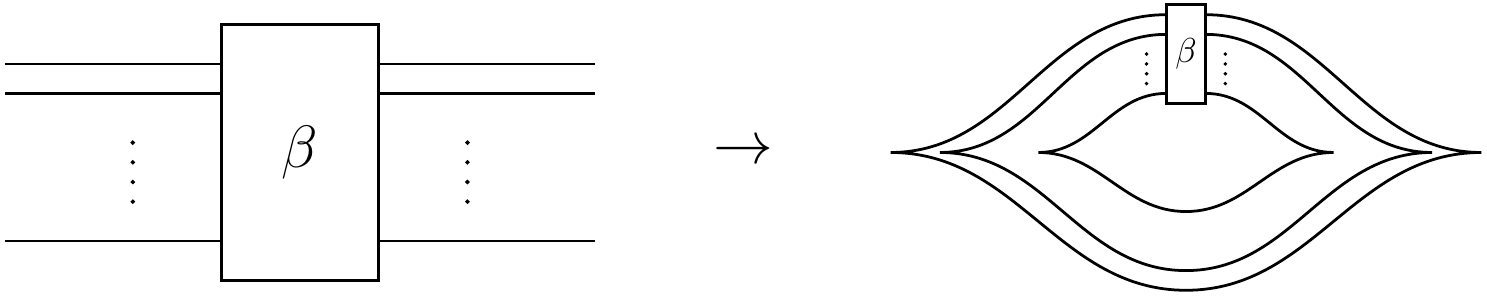}
\caption{Front diagram of a Legendrian link $\Lambda_\beta$ representing the braid closure $\hat{\beta}$ of a positive braid $\beta$.}
\label{fig:braid}
\end{figure}

Let $\xi_{\mathrm{std}}$ be the standard contact structure on $\mathbb{R}^3$ given by kernel of $\alpha_{\mathrm{std}} = dz-ydx$. Recall that a link $\Lambda \subset (\mathbb{R}^3, \xi_{\mathrm{std}})$ is \emph{Legendrian} if $T_p\Lambda\subset (\xi_{\mathrm{std}})_p$ for all  $p\in\Lambda$. A positive braid $\beta$ defines the front diagram of an oriented Legendrian link $\Lambda_\beta$ as indicated in Figure~\ref{fig:braid}. We say that $\Lambda_\beta$ is the \emph{Legendrian closure} of $\beta$ and it can be checked that $\Lambda_\beta$ is a Thurston-Bennequin-number-maximizing Legendrian representative of the  braid closure $\hat{\beta}$. Let $\mathbb{R}^4 = \mathbb{R}_t \times \mathbb{R}^3$ be the symplectization of $(\mathbb{R}^3, \xi_{\mathrm{std}})$, with symplectic form $d(e^t\alpha_{\mathrm{std}})$. An embedded surface $L\subset \mathbb{R}^4$ is \emph{Lagrangian} if $d(e^t\alpha_{\mathrm{std}})|_L\equiv 0$. 

\begin{definition}\label{Def:LagCob}
Given two Legendrian links $\Lambda_-$ and $\Lambda_+$, a \emph{Lagrangian cobordism} from $\Lambda_-$ to $\Lambda_+$ is an embedded Lagrangian $L\subset \mathbb{R}^4$ such that 
\begin{align*}
L\cap ((-\infty,-T)\times \R^3) &= (-\infty,-T)\times\Lambda_-,\\
L\cap ((T, \infty)\times \R^3) &= (T, \infty)\times \Lambda_+
\end{align*}
for some $T>0$. A Lagrangian cobordism is called \emph{exact} if there exists a function $f\colon L \to \mathbb{R}$ such that $df=\left(e^t\alpha_{\mathrm{std}}\right)|_L.$
\end{definition}

Note that the definition of Lagrangian cobordism is not symmetric. In fact, it is known that there exists a pair of Legendrian knots $\Lambda_-$ and $\Lambda_+$ where there is a genus zero Lagrangian cobordism from $\Lambda_-$ to $\Lambda_+$ but no Lagrangian cobordism from $\Lambda_+$ to $\Lambda_-$ \cite{Chantraine:2015-1}. 

It was proven by Chaintraine \cite{Chantraine:2010-1} that, if there is an exact Lagrangian cobordism from $\Lambda_-$ to $\Lambda_+$, then 
\begin{equation}\label{eqn:lagtb}
\mathit{tb}(\Lambda_+)-\mathit{tb}(\Lambda_-) = -\chi(L)\ \text{ and }\ \mathit{rot}(\Lambda_+)=\mathit{rot}(\Lambda_-),
\end{equation}
where $\mathit{tb}$ and $\mathit{rot}$ are Thurston-Bennequin number and rotation number, respectively. The following theorem provides many interesting Lagrangian cobordisms.

\begin{figure}[htbp]
\includegraphics[width=14cm]{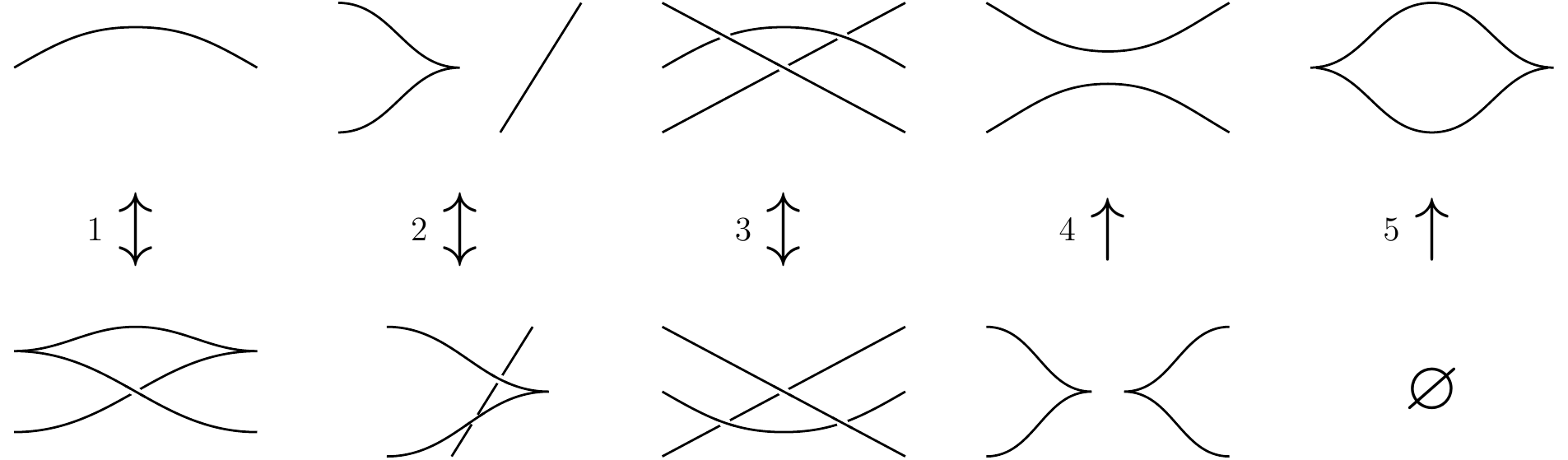}
\caption{Moves 1, 2, and 3 correspond to a Legendrian isotopy, move 4 corresponds to a pinch, and move 5 corresponds to a birth. {These moves are called \emph{elementary cobordism} (including horizontal and vertical reflections)}.}
\label{fig:lagrangian}
\end{figure}

\begin{theorem}[{\cite{Bourgeois-Sablof-Traynor:2015-1, Chantraine:2010-1, Dimitroglou:2016-1, Ekholm-Honda-Kalman:2016-1}}]\label{thm:decomp}If {diagrams of} two Legendrian links $\Lambda_-$ and $\Lambda_+$ are related by a sequence of moves shown in Figure~\ref{fig:lagrangian}, then there exists an exact Lagrangian cobordism from $\Lambda_-$ to $\Lambda_+$.\qed
\end{theorem}

Finally, we say that a Lagrangian cobordism is \emph{decomposable} if it is the result of stacking elementary cobordisms. It is not known if every exact Lagrangian cobordism is decomposable.

\begin{proof}[Proof of Lemma~\ref{lemma:posbraidstodecompcobs}]
By the definition of decomposable Lagrangian cobordism, we only need to show that each move corresponds to an elementary cobordism. For each move, this is described by one of the figures below.

\begin{figure}[htbp]
\includegraphics[width=12.44cm]{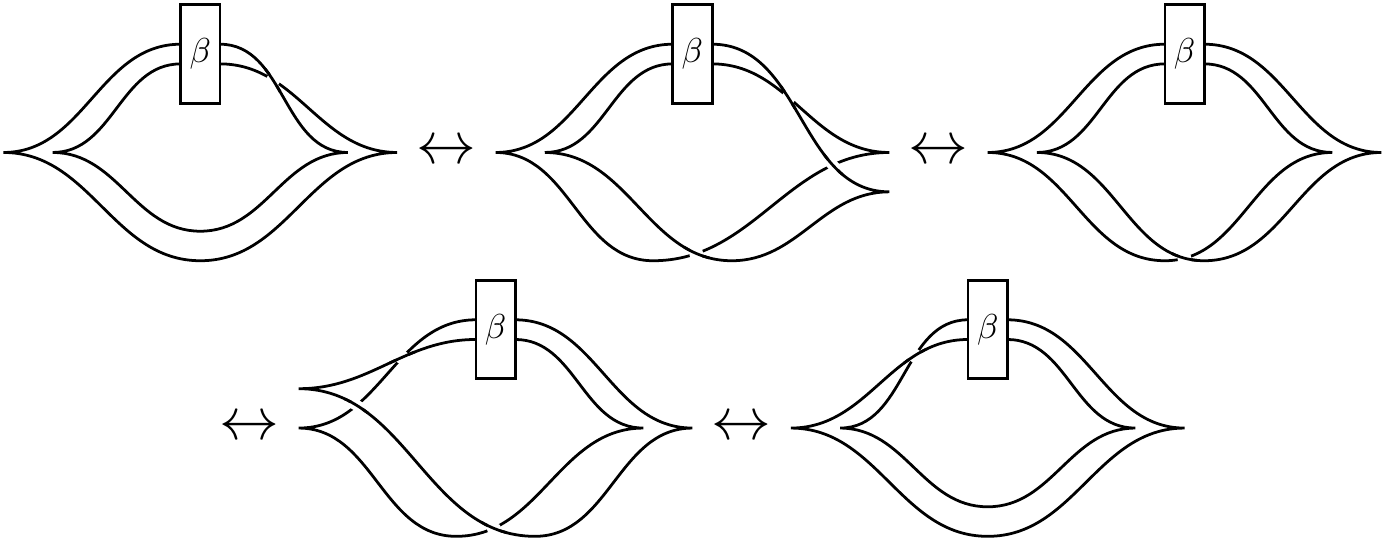}
\caption{Sequence of Legendrian isotopy to achieve a cyclic permutation. The rest of the strands are omitted.}
\label{fig:permutation}
\end{figure}

\begin{figure}[htbp]
\includegraphics[width=8cm]{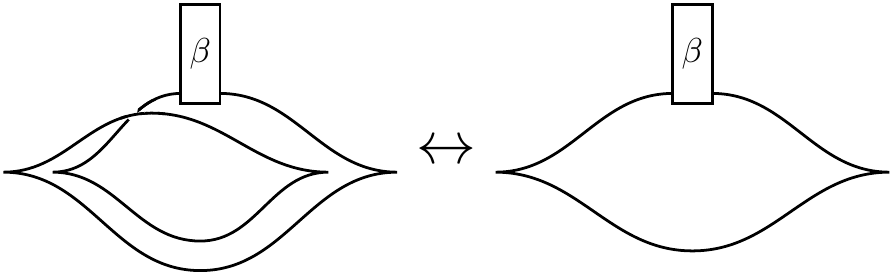}
\caption{Legendrian isotopy to achieve positive Markov stabilization. The rest of the strands are omitted.}
\label{fig:markov}
\end{figure}

\begin{figure}[htbp]
\includegraphics[width=13.33cm]{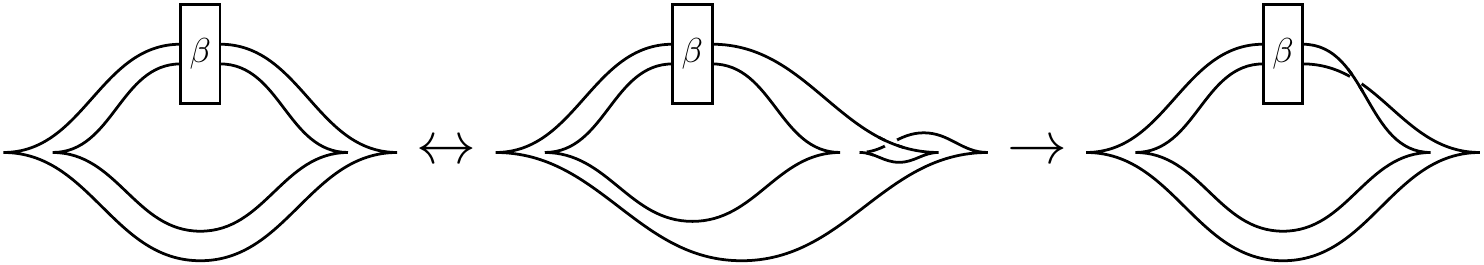}
\caption{A Legendrian isotopy and a pinch move to achieve addition of a positive Artin generator. The rest of the strands are omitted.}
\label{fig:artin}
\end{figure}

Figure~\ref{fig:permutation} shows that a cyclic permutation can be obtained by a Legendrian isotopy. Figure~\ref{fig:markov} shows that positive Markov stabilizations and destabilizations can be obtained by a Legendrian isotopy. Finally, Figure~\ref{fig:artin} shows that adding a positive Artin generator can be obtained by a Legendrian isotopy and a pinch. This completes the proof.\end{proof}

\begin{proof}[Proof of Theorem~\ref{thm:genus1deccob}]If there is a genus one exact Lagrangian cobordism from $\Lambda_{p,q}$ to $\Lambda_{p',q'}$, then there exists a genus one cobordism between $T_{p,q}$ and $T_{p',q'}$. Hence by Theorem~\ref{thm:main}, pairs which do not appear in (1)--(7) in Theorem~\ref{thm:main} do not bound genus one exact Lagrangian cobordisms.

Recall from Remark~\ref{rmk:cobordisms} that the genus one cobordisms for the pairs appearing in (1)--(5) and $\{T_{2,11}, T_{4,5}\}$ are obtained by removing two generators (which corresponds to a genus one decomposable cobordism by Lemma~\ref{lemma:posbraidstodecompcobs}), and the genus one cobordisms between the remaining five pairs of (6) are obtained by removing one generator and adding one (which corresponds to a genus one cobordism given as the concatenation of two decomposable cobordisms with first Betti number 1). Here, by concatenation we mean the operation that glues two decomposable cobordisms that do not match the directions, so their union is not necessarily Lagrangian. Therefore, we see that there exist genus one decomposable Lagrangian cobordisms for the ordered pairs appearing in (1)--(6) in Theorem~\ref{thm:genus1deccob}, and the last paragraph of the theorem follows.

Moreover, it is known that the Thurston-Bennequin number of $\Lambda_{p,q}$ is $pq-p-q$ \cite{Tanaka:1999-1}. Recall that $2g_4(T_{p,q})-1$ is also $pq-p-q$. By \eqref{eqn:lagtb}, if there is a genus one exact Lagrangian cobordism from $\Lambda_{p,q}$ to $\Lambda_{p',q'}$, then the $4$-ball genera of $T_{p,q}$ and $T_{p',q'}$ are different. Note that each pair of (6) in Theorem~\ref{thm:main} excluding $\{T_{2,11}, T_{4,5}\}$ has the same $4$-ball genus. This concludes the proof of Theorem~\ref{thm:genus1deccob}.\end{proof}

%%%%%%%%%%%%%%%%%%%%%%%%%%%%%%%%%%%%%%%%%%

\begin{appendix}
\section{Graphical representation of the results}\label{appendixA}

We end this article by summarizing Theorem~\ref{thm:main}, \ref{thm:genus1deccob}, and Corollary~\ref{cor:3} for coprime $2\leq p<q\leq 24$ with $p\leq 11$ in the following graphs.
\begin{figure}[htbp]
\includegraphics[width=14.5cm]{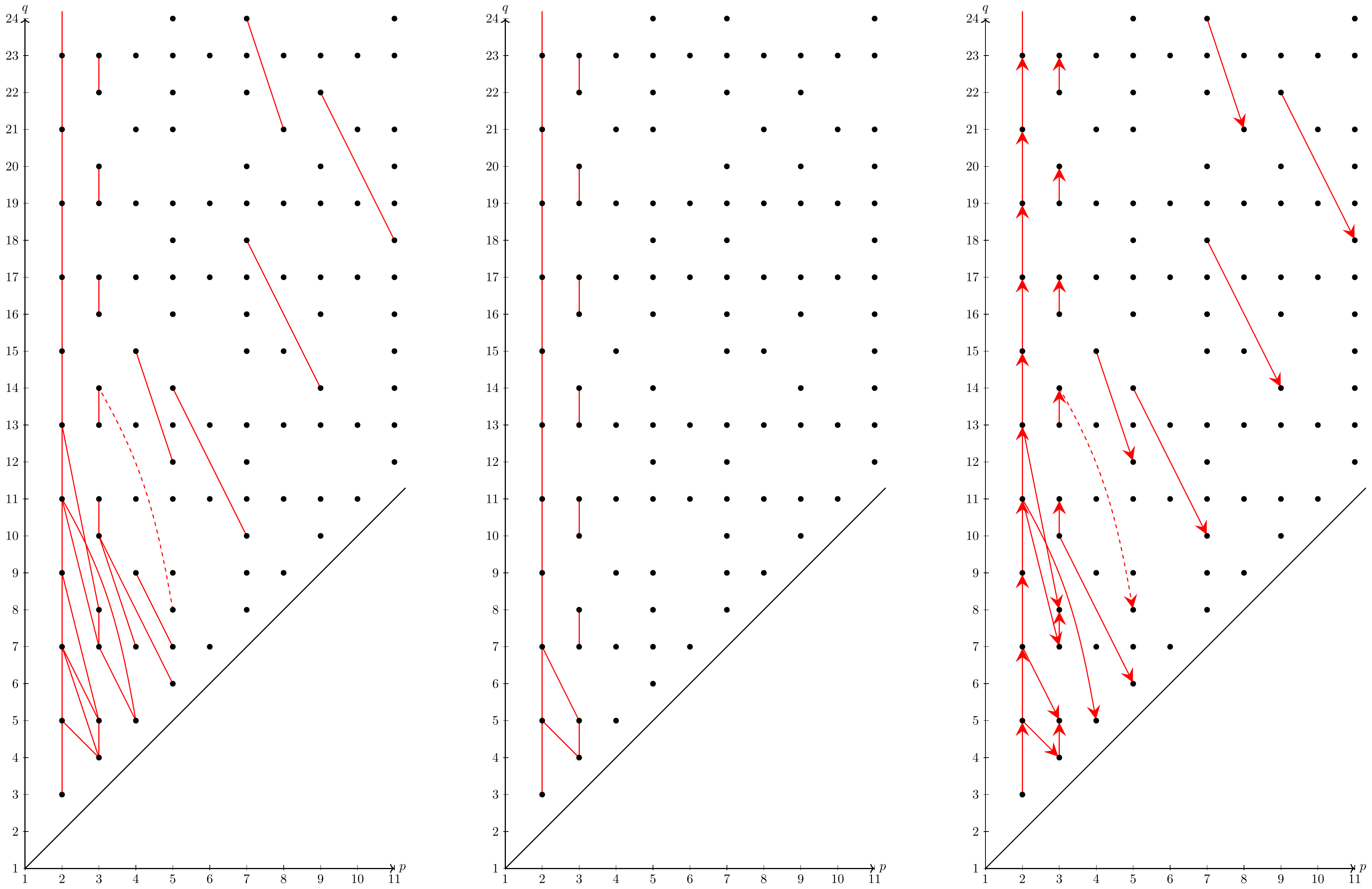}
\caption{Left: Cobordism distance one pairs. Middle: Gordian distance one pairs. Right: Pairs that have a Lagrangian cobordism of genus one between them.
\\
Dotted: The one pair for which we do not know whether there exists a genus one cobordism.}
\label{fig:graph}
\end{figure}
\end{appendix}

\bibliographystyle{alpha}
\def\MR#1{}
\bibliography{bib}
\end{document}